\documentclass[10pt]{article}
\usepackage{amsfonts}
\usepackage{amsmath}
\usepackage{amssymb}
\usepackage{amsthm}
\usepackage{bm}
\usepackage{enumitem}
\usepackage{dsfont}
\usepackage{graphicx}
\usepackage[usenames]{color}
\definecolor{red}{rgb}{1.0,0.0,0.0}

\definecolor{blu}{rgb}{0.0,0.0,1.0}

\definecolor{gre}{rgb}{0.03,0.50,0.03}

\usepackage{changes}
\definechangesauthor[name=giovanni, color=purple]{GZ}
\definechangesauthor[name=margherita, color=blue]{MZ}

\newtheorem{lemma}{Lemma}[section]
\newtheorem{theorem}[lemma]{Theorem}
\newtheorem{proposition}[lemma]{Proposition}
\newtheorem{corollary}[lemma]{Corollary}
\newtheorem{definition}[lemma]{Definition}
\newtheorem{problem}[lemma]{Problem}
\newtheorem{remark}[lemma]{Remark}
\newtheorem{hypothesis}[lemma]{Assumption}

\newcommand{\xphi}[2]{\left(\begin{smallmatrix}#1 \\ #2\end{smallmatrix}\right)}

\newcommand{\overbar}[1]{\mkern 1mu\overline{\mkern-1mu#1\mkern-1mu}\mkern 1mu}

\newcommand{\cdob}{\boldsymbol{\cdot}}

\def\sqr#1#2{{\vcenter{\vbox{\hrule height .#2pt \hbox{\vrule
 width .#2pt height#1pt \kern#1pt \vrule
width .#2pt} \hrule height .#2pt}}}}

\def\qedo{\hbox{\hskip 6pt\vrule width6pt height7pt
depth1pt  \hskip1pt}\bigskip}

\newcommand{\ind}{\mathds{1}}

\def\eps{\varepsilon}

\def\ds{\begin{displaystyle}}
\def\eds{\end{displaystyle}}

\def\<{\left\langle }
\def\>{\right\rangle }

\def\H{\mathbb H}
\def\R{\mathbb R}

\def\C{\mathbb C}
\def\E{\mathbb E}
\def\P{\mathbb P}

\def\F{\mathbb F}

\def\calb{{\cal B}}

\def\cald{{\cal D}}

\def\calf{{\cal F}}

\def\calh{{\cal H}}
\def\cali{{\cal I}}

\def\calm{{\cal M}}
\def\caln{{\cal N}}

\def\calx{{\cal X}}

\def\cals{{\cal S}}

\def\Id{{\operatorname{Id}}}

\def\1{\mathbf 1}
\def\to{\rightarrow}

\newcommand{\my}{\mu_y}

\allowdisplaybreaks

\begin{document}

\title{\bf Optimal portfolio choice with path dependent benchmarked labor income:
a mean field model}

\author{Boualem Djehiche\footnote{Department of Mathematics, KTH Royal Institute of Technology, 100 44, Stockholm, Sweden. E-mail: boualem@kth.se}, Fausto Gozzi\footnote{Dipartimento di Economica e Finanza, Luiss University, Viale Romania 32, 00197 Roma, Italy. E-mail: fgozzi@luiss.it}, Giovanni Zanco\footnote{Dipartimento di Economica e Finanza, Luiss University, Viale Romania 32, 00197 Roma, Italy. E-mail: gzanco@luiss.it} and Margherita Zanella\footnote{Dipartimento di Matematica ``Francesco Brioschi'', Politecnico di Milano, Via Bonardi 13, 20133 Milano, Italy. E-mail: margherita.zanella@polimi.it}}

\date{}

\maketitle


\begin{abstract}
We consider the life-cycle optimal portfolio choice problem faced by an agent receiving labor income and allocating her wealth to risky assets and a riskless bond subject to a borrowing constraint.
In this paper, to reflect a realistic economic setting, we propose a model where the dynamics of the labor income has two main features. First, labor income adjust slowly to financial market shocks, a feature already considered in Biffis et al. (2015) \cite{BGP}. Second, the labor income $y_i$ of an agent $i$ is benchmarked against  the labor incomes of a population $y^n:=(y_1,y_2,\ldots,y_n)$ of $n$ agents with comparable tasks and/or ranks. This last feature has not been considered yet in the literature and is faced taking the limit when $n\to +\infty$
so that the problem falls into the family of optimal control of infinite dimensional McKean-Vlasov Dynamics, which is a completely new and challenging research field.

We study the problem in a simplified case where, adding a suitable new variable, we are able to find explicitly the solution of the associated HJB equation and find the optimal feedback controls.
The techniques are a careful and nontrivial extension of the ones introduced in the previous papers of Biffis et al., \cite{BGP,BGPZ}.
\end{abstract}

\bigskip


\bigskip



\newpage

\tableofcontents


\section{Introduction}
We consider the life-cycle optimal portfolio choice problem faced by an agent receiving labor income and allocating her wealth to risky assets and a riskless bond subject to a borrowing constraint.  In line with the empirical findings and best practice, to reflect a realistic economic setting, the dynamics of labor incomes should include two main features. Firstly, labor incomes adjust slowly to financial market shocks, and income shocks have modest persistency when individuals can learn about their earning potential (see, e.g., \cite{KHAN_1997}, \cite{DICKENS_ET_AL_2007}, \cite{LEBIHAN_ET_AL_2012}). This suggests that delayed dynamics may represent a very tractable way of modelling wages that adjust slowly to financial market shocks (e.g., \cite{DYBVIG_LIU_JET_2010}, section~6). This aspect has been considered in the recent paper \cite{BGP} (see also \cite{BCGZ}, \cite{BGZ} in which the dynamics of labor income is modeled as path-dependent delayed diffusion process of the form (see Section \ref{Problem formulation} below for further details):
\begin{equation}\label{delay}
{\rm d}y(t) =\left[\my y(t)+\int_{-d}^0 \phi(s) y(t+s) {\rm d}s  \right] {\rm d}t + y(t)\sigma_y {\rm d}Z(t),
\end{equation}
where  $Z$ is a Brownian motion. The resulting optimal control problem which entails maximization of the expected power utility from lifetime consumption and bequest, subject to a linear state equation containing delay, as well as a state constraint (which is well known to make the problem considerably harder to solve), is infinite-dimensional, and can be seen as an infinite-dimensional generalization of Merton's  optimal portfolio problem.  In \cite{BGP} the authors were able to solve it completely obtaining the optimal controls in feedback form (Theorem~4.12), which
can be considered as the infinite dimensional generalization of the explicit solution to Merton's optimal portfolio problem which furthermore allows to fully understand the economic implications of the setting.

Secondly, and this is the novelty of this paper, the labor income $y_i$ of an agent $i$ is benchmarked against  the labor incomes of a population $y^n:=(y_1,y_2,\ldots,y_n)$ of $n$ agents with comparable tasks or ranks among the profession such as the level of full professor, associate professor, actuary, trader, risk manager etc., where one usually uses some wage level  $b(y^n)$ as a reference to declare whether that agent has a superior, fair or inferior labor income compared with her peers. {Typically, the labor income $y_i$ `mean-reverts' to the benchmark} $b(y^n)$ with some mean reversion speed $\epsilon$ (see e.g. \cite[\S 6]{DYBVIG_LIU_JET_2010} or \cite{BENZONI_ET_AL_2007} for the introduction of mean reverting terms in modeling labor income dynamics).  Moreover, such a benchmark $b(y^n)$ should reflect some `consensus' labor income of an indistinguishable agent within the peer group. Many corporations use the average $\bar{y}_n(t):=\frac{1}{n-1}\sum_{j=1, j\neq i}^{n} y_j(t)$, the median wage or the truncated average above a certain level $\ell$ within the company or even within the profession, $y^{\ell}_n(t):=\frac{1}{n-1}\sum_{j=1, j\neq i}^{n} y_j(t)I_{\{y_j(t)\ge \ell\}}$, used as incentive to keep attractive agents within the company,  as benchmark. These measures reflect some aggregation mechanism of some or all of the agents' labor incomes.

Let the benchmark $b(y^n)$ be the average income of a population of $n$ individuals at time $t$, $b(y^n):=\bar{y}_{n}(t):=\frac{1}{n-1}\sum_{j=1,j\neq i}^{n} y_j(t)$.  The dynamics of the $i$-th agent's labor income which includes the above mentioned aspects, i.e. path-dependency and benchmarking, can be modeled as
\begin{equation}\label{n-agents}
{\rm d}y_i(t) =\left[\epsilon(y_i(t)-\bar{y}_{n}(t))+\int_{-d}^0 \phi(s) y_i(t+s) {\rm d}s  \right] {\rm d}t + y_i(t)\sigma_y{\rm d}Z_i(t)
\end{equation}
where the $Z_i$'s are independent Brownian motions. Thus, $y^n=(y_1,y_2,\ldots,y_n)$ solves a system of interacting diffusions which are statistically indistinguishable i.e. have exchangeable joint laws. By the propagation of chaos property (see e.g. \cite{JMW}, Theorem 1.3), in the limit $n\to \infty$, the dynamics of the labor income of the representative agent is of McKean-Vlasov or mean-field type and reads
\begin{equation}\label{intro-mean}
dy(t) =\left[\epsilon(y(t)-\E[y(t)])+\int_{-d}^0 \phi(s) y(t+s) {\rm d}s  \right] {\rm d}t + y(t)\sigma_y {\rm d}Z(t).
\end{equation}
Thus, to extend the infinite dimensional  generalization of Merton's optimal portfolio problem of \cite{BGP} to benchmarked labor income dynamics, the mean field  delayed SDE  \eqref{intro-mean} can be used as the dynamics of the labor income of the representative agent instead of the system of $n$ interacting diffusions for arbitrarily large $n$ agents.

In the general case, to reflect consensus and aggregation, the benchmark wage $b(y^n)$ should be chosen such that
$y_1,y_2,\ldots,y_n$ which solve a system of interacting diffusions of the form
\begin{equation}\label{b-n-agents}
{\rm d}y_i(t) =\left[\epsilon(y_i(t)-b(\nu^n(t)))+\int_{-d}^0 \phi(s) y_i(t+s) {\rm d}s  \right] {\rm d}t + y_i(t)\sigma_y{\rm d}Z_i(t)
\end{equation}
are statistically indistinguishable ($\nu^n(t)$ denoting here the empirical measure of $y^n(t)$), in which case, in the limit $n\to \infty$, the dynamics of the representative agent's labor income satisfies the mean field type dynamics
\begin{equation}\label{intro-mf}
{\rm d}y(t) =\left[\epsilon(y(t)-b(\text{law}(y(t))))+\int_{-d}^0 \phi(s) y(t+s) {\rm d}s  \right] {\rm d}t + y(t)\sigma_y {\rm d}Z(t).
\end{equation}
It follows that the resulting optimal control problem adds a `mean-field aspect' to the infinite-dimensional generalization of Merton's optimal portfolio problem studied in \cite{BGP}.
{To be precise this problem falls into both families studied in this area:
\begin{itemize}
\vspace{-0.2truecm}
\item  the `Mean-Field Games' where we look for a Nash equilibrium of a game with many players;
\item
\vspace{-0.2truecm}
the `Optimal Control of McKean-Vlasov Dynamics', where a unique representative agent (the `planner') takes the decisions.
\end{itemize}
\vspace{-0.2truecm}
In general the above two problems are different and give different results (see e.g. \cite[\S 6.1]{CarmonaDelarue18}) but in our case, since the labor income is not influenced by the choice of the agents, they turn out to be the same. This means that the results of this paper can be interpreted under different angles. Here our goal is mainly to develop the theoretical machinery to find the solution while we leave the
analysis of its financial consequences for a subsequent paper.}

In this paper we explicitly solve the path-dependent generalization of Merton's  optimal portfolio problem under the labor income dynamics \eqref{intro-mean}. We are able to do this using a suitable infinite-dimensional general  problem (coming from a change of variable introduced in Remark \ref{rm:changevar}) whose associated HJB equation admits an explicit solution $\tilde v$ which allows to find the optimal control strategies in an explicit feedback/ closed-loop form.

Under the general dynamics \eqref{intro-mf}, an explicit solution of the associated infinite dimensional HJB equation is however out of reach. Even establishing existence and uniqueness and deriving qualitative properties of the solution of the associated HJB seems a hard problem to solve for the time being.
The main issue here is that, even in the finite-dimensional non-path-dependent case,
the theory for HJB equations arising in the optimal control of McKean-Vlasov dynamics is at a very initial stage: only few results on viscosity solutions are available and no regularity theorem is proved up to now, except in very specific settings, like the linear quadratic one.
Concerning the finite-dimensional non-path-dependent, one can see e.g.  the book \cite{CarmonaDelarue18} for an account of the theory, and the papers \cite{BurzoniEtAl19,CossoPham19,PhamWei17} for some recent results.
Concerning instead the finite-dimensional path-dependent case one can see the paper \cite{WuZhang20} for some results on viscosity solutions of the HJB equations.
Finally, up to now, concerning mean-field games in infinite dimension, we only know the linear quadratic model of \cite{FouqueZhang18}.

\smallskip

The structure of the paper is as follows.

\vspace{-0.2truecm}

\begin{itemize}
  \item
In Section \ref{Problem formulation} we
outline the model and, in Remark \ref{rm:changevar}, introduce the change of variable which we use to rewrite it in a more treatable form.

\vspace{-0.2truecm}

  \item
Section \ref{sec:constraint} is devoted to the non-trivial task of rewriting the no-borrowing constraint (see \eqref{NO_BORROWING_WITHOUT_REPAYMENT_CONDITIONLA_MEAN} below) in our case.

\vspace{-0.2truecm}

  \item
In Section \ref{SEC:HJB} we first write our general  problem (Problem \ref{pbl}) in a suitable infinite-dimensional setting (Subsection \ref{sub:pbinfdim}). Then, in Subsection \ref{sub:HJB}, we write and solve the associated
HJB equation (Theorem \ref{thm_sol1}).

\vspace{-0.2truecm}

\item
In Section 5, we solve the general  problem.
First, in Subsection 5.1, we provide a lemma to understand what happens to admissible strategies when the boundary of the constraint set is reached, a key feature in dealing with state constraints problems.
Then, in Subsections 5.2-5.3, we prove the fundamental identity and the verification theorem, which allow us to find the optimal strategies in feedback form. Here, for brevity, we consider mainly the case $\gamma \in (0,1)$ simply recalling how to deal with the case $\gamma >1$.

\vspace{-0.2truecm}

\item
Finally, Section 6 summarizes the main results of the paper for the original problem, with a short discussion.
\end{itemize}

\vspace{-0.2truecm}

\section{Problem formulation}\label{Problem formulation}

We begin with the basic setting which is borrowed from \cite{DYBVIG_LIU_JET_2010} and \cite{BGP} and is repetead here for the reader's convenience.

Consider a filtered probability space $(\Omega, \mathcal F, \mathbb F, \mathbb P)$, where we define the $\mathbb F$-adapted vector valued process $(S_0,S)$ representing the price evolution of a riskless asset, $S_0$, and $n$ risky assets, $S=(S_1,\ldots,S_n)^\top$, with dynamics
\vspace{-0.2truecm}
\begin{eqnarray}\label{DYNAMIC_MARKET}
\left\{\begin{array}{ll}
dS_0(t)= S_0(t) r  dt\\
dS(t) =\text{diag}(S(t)) \left(\mu dt + \sigma dZ(t)\right)\\
S_0(0)=1\\
S(0)\in {\mathbb R}^n_{+},
\end{array}
\right.
\vspace{-0.2truecm}
\end{eqnarray}
where we assume the following.
\begin{hypothesis}\label{hp:S}
\begin{itemize}
  \item[]
  \item[(i)]
\vspace{-0.2truecm}
$Z$ is a $n$-dimensional Brownian motion.
The filtration $\mathbb F=(\mathcal F_t)_{t \ge 0}$ is the one generated by $Z$, augmented with the $\P$-null sets.
\vspace{-0.2truecm}
  \item[(ii)] $\mu \in \mathbb R^n$, and the matrix $\sigma \in  \mathbb R^{n \times n}$ is invertible.
\end{itemize}
\end{hypothesis}


An agent is endowed with initial wealth $w\ge 0$, and receives
labor income $y$ until the stopping time $\tau_{\delta}>0$, which represents the agent's random time of death.
We assume the following.
\begin{hypothesis}\label{hp:tau}
\begin{itemize}
  \item[]
  \item[(i)]
\vspace{-0.2truecm}
$\tau_{\delta}$ is independent of $Z$, and it has exponential law with parameter $\delta >0$.
\vspace{-0.2truecm}
  \item[(ii)] The reference filtration is accordingly
given by the enlarged filtration {\color{black} $\mathbb G := \big( \mathcal G_t \big)_{t \ge 0}$,}
where each sigma-field {\color{black} $\mathcal G_t$}
is defined as
\vspace{-0.2truecm}
\begin{equation*}
	{\color{black} \mathcal G_t}:= \cap_{u>t} \left(\mathcal F_u \vee  \sigma_g\left(\tau_{\delta}\wedge u\right)\right),
\vspace{-0.2truecm}
\end{equation*}
augmented with the $\mathbb P$-null sets. Here by $\sigma_g (U)$ we denote the sigma-field generated by the random variable $U$.
\end{itemize}
\end{hypothesis}
Note that, with the above choice, $\mathbb G$ is the minimal enlargement of the Brownian filtration satisfying the usual assumptions and making $\tau_\delta$ a stopping time {\color{black}
(see \cite[Section VI.3, p.370]{Protter} or \cite[Section 7.3.3, p.420]{JYC}).
Moreover, see \cite[Proposition 2.11-(b)]{AKSAMITJEANBLANC17}, we have the following result.
If a process $A$ is ${\mathbb G}$-predictable then there exists a process $a$
which is ${\mathbb F}$-predictable and such that
\vspace{-0.2truecm}
\begin{equation}\label{eq:GFpred}
A(s,\omega) =  a(s,\omega),\qquad
\forall \omega\in \Omega,\; \forall s \in [0,\tau_\delta(\omega)].
\vspace{-0.2truecm}
\end{equation}

Therefore, arguing as in \cite[Section 2]{BGP}, we can reduce the problem (which is initially relative to the larger filtration $\mathbb G$) to the ``pre-death'' one (where we work with $\mathbb F$-predictable processes). Hence, from now on we express the problem in terms of $\mathbb F$-predictable processes.

The agent can invest her resources in the riskless and risky assets, and can consume at rate $c(t)\geq 0$. We denote by $\theta(t)\in \mathbb R^n$ the amounts allocated to the risky assets at each time $t\geq 0$. The agent can also purchase life insurance to reach a bequest target $B(\tau_\delta)$ at death, where $B(\cdot)\geq 0$ is also chosen by the agent. We let the agent pay an insurance premium of amount $\delta(B(t)-W(t))$ to purchase coverage of face value $B(t)-W(t)$, for $t<\tau_\delta$. As in \cite{DYBVIG_LIU_JET_2010}, we interpret a negative face value $B(t)-W(t)<0$ as a  life annuity trading wealth at death for a positive income flow $\delta(W(t)-B(t))$ while living. The controls $c, \theta$, and $B$  are for the moment assumed to belong to the following set:
\vspace{-0.2truecm}
\begin{multline}\label{DEF_PI0_FIRST_DEFINITION}
\Pi^0:= \Big\{\mathbb F-\mbox{predictable} \ c(\cdot), B(\cdot), \theta(\cdot) \colon c(\cdot), B(\cdot) \in L^1 (\Omega \times [0, +\infty);\mathbb R_{+}),\theta(\cdot) \in L^2(\Omega \times \mathbb R; \mathbb R^n)\Big\}.
\vspace{-0.2truecm}
\end{multline}

The agent's wealth (before death) is assumed to obey to the standard dynamic budget constraint of the Merton portfolio model, but with the labor income 
	and insurance premium 
	terms  added in the drift, exactly as in \cite{DYBVIG_LIU_JET_2010} and \cite{BGP}.
On the other hand the evolution of the labor income $y$ here is new. The main novelty here is that,
as opposed to standard bilinear SDEs (as in, e.g., \cite{DYBVIG_LIU_JET_2010}) and to bilinear path-dependent SDEs (as in \cite{BGP}),
we assume the labor income $y$ to follow a bilinear SDE where the drift contains not only a path-dependent term but also a mean reverting term.
Hence, the dynamics of the state variables $(W,y)$ are as follows:
\begin{align}\label{DYNAMICS_WEALTH_LABOR_INCOME}
\begin{split}
\left\{\begin{array}{ll}
dW(t) = & \left[W(t) r + \theta(t)\cdob (\mu-r\mathbf{1})  + y(t) - c(t)
-\delta\left(B(t)-W(t)\right)\right] {\rm d}t + \theta(t)\cdob\sigma {\rm d}Z(t) \\[2mm]
dy(t) = & \left[\epsilon(y(t)-\E[y(t)])+\my y(t)+\int_{-d}^0 \phi(s) y(t+s) {\rm d}s  \right] {\rm d}t + y(t)\sigma_y\cdob {\rm d}Z(t),\\[2mm]
W(0) = & w,\\
y(0)= & x_0, \quad y(s) = x_1(s) \mbox{ for $s \in  [-d,0)$},
\end{array}\right. \end{split}
\end{align}
where $(c,B,\theta)\in \Pi^0\left(w,x_0,x_1\right)$, $\mu_y,\epsilon \in  \mathbb R$, $\sigma_y \in  \mathbb R^n$, $\mathbf 1 = (1,\dots, 1)^\top$ is the unitary vector in $\mathbb R^n$, $\cdob$ denotes the canonical inner product of $\R^n$ and the functions $\phi(\cdot), x_1(\cdot)$ belong to $L^2\left(-d,0; \mathbb R\right)$.

\begin{remark} From an economic point of view, as in \cite{BGPZ}, the term $\mu_y y(t)$ in the dynamics of $y$ in \eqref{DYNAMICS_WEALTH_LABOR_INCOME} models a discounting effect at rate $\mu_y$ to account for a possible inflationary ($\mu_y <0$)/deflationary ($\mu_y >0$) regime.
Moreover, in the mean reverting term, it is standard to choose $\epsilon<0$. Here we take generic $\epsilon\in \R$ since our method of solution works also in this case.
\end{remark}

Once the control strategies $(c,B,\theta)\in \Pi^0\left(w,x_0,x_1\right)$ are fixed and the process $y\in L^1 (\Omega \times [0, +\infty);\mathbb R_{+})$ is given, existence and uniqueness of a strong solution to the SDE for $W$ are ensured, e.g.,
by the results of \cite[Chapter 5.6]{KARATZAS_SHREVE_91}.

On the other hand existence and uniqueness of a solution for the equation for $y$ is more delicate. When $\epsilon =0$, \cite[Theorem I.1 and Remark I.3(iv)]{MOHAMMED_BOOK_96} ensure existence and uniqueness of a solution with $\P$-a.s. continuous paths.
The case when $\epsilon\ne 0$ can be treated as in \cite{JMW} when there is no path-dependency, while the present case can be treated similarly to \cite[Subsection 5.1]{WuZhang20}}.

\begin{remark}\label{rm:changevar}
The equation for $y$ can be rewritten by introducing the new variable
\vspace{-0.2truecm}
$$
e(t)=\E[y(t)].
\vspace{-0.2truecm}
$$
Taking expectation in the equation for $y$ above we get that $e$ satisfies the delay equation
\begin{equation}\label{eq:efirst}
de(t)= \left[\my e(t)+\int_{-d}^0 \phi(s) e(t+s) {\rm d}s\right] {\rm d}t,
\end{equation}
while the equation for $y$ becomes
\begin{equation}\label{eq:yefirst}
dy(t) = \left[\epsilon(y(t)-e(t))+\my y(t)+\int_{-d}^0 \phi(s) y(t+s) {\rm d}s  \right] {\rm d}t + y(t)\sigma_y\cdob {\rm d}Z(t).\\[2mm]
\end{equation}
Now, thanks to \cite[Theorem I.1 and Remark I.3(iv)]{MOHAMMED_BOOK_96}
the system made of (\ref{eq:efirst})-(\ref{eq:yefirst}) admits a unique strong solution with $\P$-a.s. continuous paths for $t\geq 0 $ for every initial datum and it is not difficult
to prove that, when the initial data are chosen so that
$$
y(0)=  x_0, \quad y(s) = x_1(s) \mbox{ for $s \in  [-d,0)$},
$$
$$
e(0)= \E[x_0]=x_0, \quad e(s) =\E[ x_1(s)]=x_1(s) \mbox{ for $s \in  [-d,0)$},
$$
then the component $y$ of such solution is also a strong solution of the second equation of \eqref{DYNAMICS_WEALTH_LABOR_INCOME} and that, vice versa, given a strong solution to the labor income equation in \eqref{DYNAMICS_WEALTH_LABOR_INCOME}, the couple $(\E[y],y)$ solves the system \eqref{eq:efirst}--\eqref{eq:yefirst}.
\\
Hence, system \eqref{DYNAMICS_WEALTH_LABOR_INCOME}
can be rewritten, in the variables $(W,y,e)$ as
\begin{align}\label{DYNAMICS_WEALTH_LABOR_INCOMEe}
\begin{split}
\left\{\begin{array}{ll}
dW(t) = & \left[W(t) r + \theta(t)\cdob (\mu-r\mathbf{1})  + y(t) - c(t)
-\delta\left(B(t)-W(t)\right)\right] {\rm d}t + \theta(t)\cdob \sigma {\rm d}Z(t) \\[2mm]
dy(t) = & \left[\epsilon(y(t)-e(t))+\my y(t)+\int_{-d}^0 \phi(s) y(t+s) {\rm d}s  \right] {\rm d}t + y(t)\sigma_y\cdob {\rm d}Z(t),\\[2mm]
de(t) = & \left[\my e(t)+\int_{-d}^0 \phi(s) e(t+s) {\rm d}s  \right] {\rm d}t ,\\[2mm]
W(0) = & w,\\
y(0)= & x_0, \quad y(s) = x_1(s) \mbox{ for $s \in  [-d,0)$},\\
e(0)=& \E[x_0]=x_0, \quad e(s) =\E[ x_1(s)]=x_1(s) \mbox{ for $s \in  [-d,0)$}.
\end{array}\right. \end{split}
\end{align}
We will refer to this system in the sequel.
\hfill\qedo
\end{remark}


We aim to maximize the expected utility from lifetime consumption and bequest,
\begin{eqnarray}\label{DEF_OBJECTIVE FUNCTION_DEATH TIME}
\mathbb E \left(\int_{0}^{\tau_{\delta}} e^{-\rho t }
\frac{c(t)^{1-\gamma}}{1-\gamma} dt
+ e^{-\rho \tau_{\delta} } \frac{\big(k B(\tau_\delta)\big)^{1-\gamma}}{1-\gamma}
\right),
\end{eqnarray}
over all triplets $\left(c,\theta,B\right)\in \Pi^0$
satisfying a suitable no-borrowing state constraint introduced below in
\eqref{NO_BORROWING_WITHOUT_REPAYMENT_CONDITIONLA_MEAN}.
In the above, $k>0$ measures the intensity of preference for leaving a bequest, $\gamma \in (0,1) \cup (1, +\infty)$ is the risk-aversion coefficient and $\rho >0$ is the discount rate.
As the death time is independent of $Z$ and exponentially distributed, we can rewrite
the objective functional as follows  (e.g., \cite[Section 2]{BGP} or \cite[Section 3.6.2]{PHAM_BOOK_2009}):
\begin{equation}
\label{OBJECTIVE_FUNCTION}
J(c,B):=\mathbb E \left(\int_{0}^{+\infty} e^{-(\rho+ \delta) t }
\left( \frac{c(t)^{1-\gamma}}{1-\gamma}
+ \delta \frac{\big(k B(t)\big)^{1-\gamma}}{1-\gamma}\right) {\rm d}t
\right).
\end{equation}
The announced state constraint is the same as in \cite{BGP}, which is also considered in \cite{DYBVIG_LIU_JET_2010}. We present it here for the reader's convenience.
First of all recall that, given the financial market described by \eqref{DYNAMIC_MARKET}, the pre-death state-price density of the agent obeys the stochastic differential equation
\begin{equation}\label{DYN_STATE_PRICE_DENSITY}
\left\{\begin{array}{ll}
d \xi (t)& = - \xi(t)(r +\delta) {\rm d}t  -\xi(t) \kappa\cdob {\rm d}Z(t),\\
\xi(0)&=1.
\end{array}\right.
\end{equation}
where $\kappa$ is the market price of risk and is defined as follows (e.g., \cite{KARATZAS_SHREVE}):
\begin{equation}\label{DEF_KAPPA}
\kappa:= (\sigma)^{-1} (\mu- r \mathbf 1).
\end{equation}
We require the agent to satisfy the following constraint
\begin{equation}\label{NO_BORROWING_WITHOUT_REPAYMENT_CONDITIONLA_MEAN}
W(t) +   \xi^{-1}(t)\mathbb E\left( \int_t^{+\infty} \xi(u) y(u) {\rm d}u \Bigg\vert \mathcal F_t\right)  \geq 0,
\end{equation}
which is a no-borrowing-without-repayment constraint
as the second term in
\eqref{NO_BORROWING_WITHOUT_REPAYMENT_CONDITIONLA_MEAN}
represents the agent's market value of
human capital at time $t$. In other words, human capital can be pledged as collateral,
and represents the agent's maximum borrowing capacity.
We note that the agent cannot default on his/her debt upon death, as the bequest target $B$ is nonnegative.

Let us denote by $\left(W^{w,x_0,x_1}\left(t; c,B,\theta\right),y^{x_0,x_1}(t)\right)$
the solution at time $t$ of system \eqref{DYNAMICS_WEALTH_LABOR_INCOME},
where we emphasize the dependence of the solution on the initial conditions $(w,x_0,x_1)$
and strategies $(c,B,\theta)$. 
We can then define the set of admissible controls as follows:
\begin{equation}\label{DEF_PI_FIRST_DEFINITION}
\begin{split}
\Pi\left(w,x_0,x_1\right):= \Bigg\{ & c(\cdot), B(\cdot), \theta(\cdot)
\in \Pi^0\left(w,x_0,x_1\right), \ \mbox{such that:}\\
 & W^{w,x_0,x_1}\left(t; c,B,\theta\right) +  \xi^{-1}(t)  \mathbb E\left( \int_t^{+\infty} \xi(u) y^{x_0,x_1}(u) {\rm d}u \Big\vert  \mathcal F_t\right)  \geq 0\,\quad \forall t \geq 0\Bigg\}.
\end{split}
\end{equation}
Our problem is then to maximize the functional given in \eqref{OBJECTIVE_FUNCTION}
over all controls in $\Pi\left(w,x_0,x_1\right)$.




We introduce two assumptions that will hold throughout the whole paper:
\begin{hypothesis}
  \label{Hyp_K}
\begin{equation*}
  \begin{cases}
    r+\delta-(\epsilon+\my-\sigma_y\cdob\kappa)-\int_{-d}^0e^{(r+\delta)s}\left\vert\phi(s)\right\vert{\rm d}s>0&\text{ if }\epsilon-\sigma_y\cdob\kappa <0,\\
    r+\delta-\my-\int_{-d}^0e^{(r+\delta)s}\left\vert\phi(s)\right\vert{\rm d}s>0&\text{ if }\epsilon-\sigma_y\cdob\kappa\geq 0.
  \end{cases}
\end{equation*}
\end{hypothesis}

\begin{remark} $ $
\begin{itemize}
\item [$(a)$] Assumption \ref{Hyp_K} is needed
to rewrite in a convenient way (as we do in Section \ref{sec:constraint}) the constraint \eqref{NO_BORROWING_WITHOUT_REPAYMENT_CONDITIONLA_MEAN},
and will be carefully explained in Subsection \ref{subsec:comment}.
Here we only observe that this condition is a refinement of the one provided in \cite[Hypothesis 2.4]{BGP} and that in the interesting case $\epsilon <0$, only the first formula holds, which reduces, when $\eps=0$, to Hypothesis 2.4 of \cite{BGP}.

\item [$(b)$]
Mimicking the method of the proof of Proposition 2.2 in \cite{BGP}, we obtain the following representation of the labor income process $y$:
\begin{equation*}
y(t)=E(t)(x_0+I(t))
\end{equation*}
where
$$
\begin{array}{lll}
E(t)=e^{(\epsilon+\my-\frac{1}{2}\vert\sigma_y\vert^2)t+\sigma_y\cdob Z(t)}, \\
I(t)=\int_0^t E^{-1}(u)\left(-\epsilon e(u)+\int_{-d}^0\phi(s)y(s+u)ds\right)du.
\end{array}
$$
Moreover, if $x_0>0, x_1\ge 0$ a.s., $\phi\ge 0$ a.e. and $\epsilon<0$, then it must be that $y(t)>0\,\, \mathbb{P}$-a.s..
\end{itemize}
\end{remark}

\begin{hypothesis}
  \label{Hyp_gamma}
  \begin{equation*}
    \rho+\delta-(1-\gamma)\left(r+\delta+\frac{\vert\kappa\vert^2}{2\gamma}\right)>0\ .
  \end{equation*}
\end{hypothesis}

\begin{remark}\label{rm:hp2}
Assumption \ref{Hyp_gamma} is required to ensure that the candidate solution of our HJB equation in Section \ref{SEC:HJB} is well defined and finite, see Theorem \ref{thm_sol1}.
In similar simple cases it can actually be proved that,
when $\gamma\in (0,1)$ and
$$
\rho + \delta -(1-\gamma) (r + \delta +\frac{\vert\kappa\vert^2}{2\gamma }) <0,
$$
the value function is infinite; for example, see \cite{FreniGozziSalvadori06} for the deterministic case.
\end{remark}

\section{Reformulation of the constraint}
\label{sec:constraint}
Within this section we assume that the second equation of (\ref{DYNAMICS_WEALTH_LABOR_INCOME}) has a unique continuous $\F$-adapted solution $y$.\\
We find an equivalent expression for the human capital defined as
\begin{equation}
\label{def_HC}
  HC(t_0):= \xi(t_0)^{-1} \mathbb{E}\left[ \int_{t_0}^{\infty}\xi(u)y(u)\,{\rm d}u |\mathcal{F}_{t_0}\right],
\end{equation}
which is the secon summand in the left hand side of the constraint (\ref{NO_BORROWING_WITHOUT_REPAYMENT_CONDITIONLA_MEAN}).\\

Following the idea of \cite{BGPZ}, we incorporate the discount factor $\xi$ in an equivalent probability measure $\tilde{\mathbb{P}}$ (Subsection \ref{subsec:equiv_P}) and rewrite the dynamics of $y$ under $\tilde{\mathbb{P}}$ in a suitable Hilbert space, using the so-called \emph{product-space framework} for path-dependent equations (Subsection \ref{subsec:Hilbert}). Exploiting some spectral properties of the operators that appear in this formulation we can finally obtain the mentioned equivalent expression for $HC(t_0)$ (Subsections \ref{subsec:spectral} and \ref{subsec:formula_HC}). We then comment on the relation between the spectral properties used herein and our Assumption \ref{Hyp_K} (Subsection \ref{subsec:comment}).

\subsection{Equivalent probability measure}
\label{subsec:equiv_P}
We start by considering the equivalent probability measure $\tilde \P_s$ on $\mathcal{F}_s$ such that
\begin{equation}
\label{Ptilde}
  \frac{\text{d}\tilde{\mathbb  P}_s}{\text{d} \mathbb P}=\exp\left(  -\frac{1}{2} |\kappa|^2  s   - \kappa\cdob Z(s)\right)     =  e^{(r+\delta)s}  \xi(s)\ ;
\end{equation}
by \cite[Lemma 3.5.3]{KARATZAS_SHREVE_91} we can write   \[   \mathbb E   \left[ \xi(s)  y(s) \mid \mathcal F_{t_0}  \right] = \xi(t_0) e^{-(r+\delta)(s-t_0)}  \tilde{\mathbb E}_s \left[  y(s) \mid \mathcal F_{t_0}  \right]  .\]

We are reduced to evaluate
  \begin{equation}
    \label{EXPRESSION_II}
    \begin{aligned}
    \mathbb E \left[ \int_{t_0}^{+\infty}   \xi(s)  y(s)  \text{d}s \mid \mathcal F_{t_0}    \right]&=\int_{t_0}^{+\infty}  \mathbb E \left[ \xi(s)  y(s)\mid \mathcal F_{t_0}   \right]\text{d}s\\
    &=\xi(t_0) e^{(r+\delta)t_0} \int_{t_0}^{+\infty}  e^{-(r+\delta)s}  \tilde{\mathbb E}_s \left[  y(s)\mid \mathcal F_{t_0}   \right] \text{d}s.
  \end{aligned}
\end{equation}

The idea is now to understand what kind of SDE the quantity $\tilde{\mathbb E}\left[  y(s)\mid \mathcal F_{t_0}   \right]=\tilde\E_s\left[  y(s)\mid \mathcal F_{t_0}   \right]$ satisfies.
Let $\tilde{\mathbb P}$ the measure such that $\left.\tilde{\mathbb P}\right|_{\mathcal F_s}=\tilde{\mathbb P}(s)$ for all $s\geq 0$.
By the Girsanov Theorem the process
$\tilde{Z}(t) = Z(t) + \kappa t$
is an $n$-dim. Brownian motion under $ \tilde{\mathbb  P}$.
The dynamics of $y$ under $ \tilde{\mathbb  P}$ is then
\begin{align}
\begin{split}
\left\{\begin{array}{ll}
dy(t) = & \left[ (\epsilon+\my-\sigma_y\cdob \kappa)y(t)-\epsilon e(t)+\int_{-d}^0 \phi(s) y(t+s) {\rm d}s  \right] {\rm d}t + y(t)\sigma_y\cdob    {\rm d}\tilde Z(t),\\[2mm]
de(t) = & \left[\my e(t)+\int_{-d}^0 \phi(s) e(t+s) {\rm d}s  \right] {\rm d}t ,\\[2mm]
y(0)= & x_0, \quad y(s) = x_1(s) \mbox{ for $s \in  [-d,0)$},\\
e(0)=& \E[x_0]=x_0, \quad e(s) =\E[ x_1(s)]=x_1(s) \mbox{ for $s \in  [-d,0)$}.
\end{array}\right. \end{split}
\end{align}
\begin{remark}
Under the equivalent probability measure $\tilde{\mathbb{P}}$
the DDE satisfied by $e$ remains the same.
\end{remark}

Therefore, the quantity $\tilde{\mathbb{E}}\left[y(t)|\mathcal{F}_{t_0}\right]$ satisfies the equation
\begin{align*}
\tilde{\mathbb{E}}\left[y(t)|\mathcal{F}_{t_0}\right]
&=y(t_0)+ (\epsilon+\my-\sigma_y\cdob \kappa) \int_{t_0}^t \tilde{\mathbb{E}}\left[y(s)|\mathcal{F}_{t_0}\right]\, {\rm d}s
- \epsilon\int_{t_0}^t \tilde{\mathbb{E}}\left[e(s)|\mathcal{F}_{t_0}\right]\, {\rm d}s
\\
&\qquad+ \int_{t_0}^t \int_{-d}^0 \tilde{\mathbb{E}}\left[y(s+\tau)|\mathcal{F}_{t_0}\right]\phi(\tau)\, {\rm d} \tau\, {\rm d}s
\\
&=y(t_0)+ (\epsilon+\my-\sigma_y\cdob \kappa) \int_{t_0}^t \tilde{\mathbb{E}}\left[y(s)|\mathcal{F}_{t_0}\right]\, {\rm d}s
-  \epsilon\int_{t_0}^t e(s)\, {\rm d}s
\\
&\qquad+ \int_{t_0}^t \int_{-d}^0 \tilde{\mathbb{E}}\left[y(s+\tau)|\mathcal{F}_{t_0}\right]\phi(\tau)\, {\rm d} \tau\, {\rm d}s
,
\end{align*}
where in the last equality we exploit the fact that $e$ satisfies a deterministic equation, thus $\tilde{\mathbb{E}} [e(s)|\mathcal{F}_{t_0}]=e(s)$.
Notice that the stochastic integral with respect 
to $\tilde{Z} $ is a martingale, and has zero mean, hence $\tilde{\mathbb{E}}\left[\int_{t_0}^ty(s)\sigma_y\cdob{\rm d}\tilde{Z}(s)|\mathcal{F}_{t_0}\right]=0$ (for more details see \cite[Lemma 4.6]{BGPZ}).

Therefore, defining $M_{t_0}(t):=\tilde{\mathbb{E}}\left[y(t)|\mathcal{F}_{t_0}\right]$, we have that the couple $(M_{t_0},e)$ satisfies for $t\geq t_0$ the system (with random initial conditions)
\begin{align}
\label{system1}
\begin{cases}
{\rm d}M_{t_0}(t)=\left[(\epsilon+\my- \sigma_y\cdob \kappa)M_{t_0}(t)-\epsilon e(t)+ \int_{-d}^0 M_{t_0}(t+s)\phi(s)\, {\rm d}s\right]{\rm d}t
\\
{\rm d}e(t)=\my e(t){\rm d}t + \int_{-d}^0e(t+s)\phi(s)\,{\rm d}s{\rm d}t
\\
M_{t_0}(t_0)=y(t_0), \quad M_{t_0}(t_0+s)=y(t_0+s), \qquad s \in[-d,0],
\\
e(t_0)=\E\left[y(t_0)\right], \quad e(s)=\E\left[y(t_0+s)\right], \qquad s \in [-d,0].
\end{cases}
\end{align}
\subsection{Reformulation of the problem in an infinite-dimensional framework}
\label{subsec:Hilbert}
We introduce first the space
\begin{equation*}
  \calm_2=\R\times L^2(-d,0;\R)
\end{equation*}
whose elements are denoted as $\overbar x=(x_0,x_1)$. $\calm_2$ is a Hilbert space when endowed with the inner product $\langle (x_0,x_1),(y_0,y_1)\rangle_{\calm_2}=x_0y_0+\langle x_1,y_1\rangle$, the latter being the usual inner product of $L^2(-d,0;\R)$.\\
We will denote vectors in $\R^2$ with boldface letters: $\mathbf{a}=\xphi{b}{c}$; similarly we will write $\R^2$-valued functions as $\mathbf{f}(\cdot)=\xphi{g(\cdot)}{h(\cdot)}$.\\
The state space for the reformulation needed within this section is $\mathcal{M}_2^2=\left(\R\times L^2(-d,0;\R)\right)^{\oplus 2}\cong \mathbb{R}^2 \times L^2(-d,0;\mathbb{R}^2)$. Elements of $\calm_2^2$ will be written in any of the following equivalent ways:
\begin{equation}
  \label{notation1}
  \overbar{\mathbf{x}}=\begin{pmatrix}(x^{(1)}_0,x^{(1)}_1)\\(x^{(2)}_0,x^{(2)}_1)\end{pmatrix}\text{ with }(x^{(1)}_0,x^{(1)}_1),(x^{(2)}_0,x^{(2)}_1)\in\calm_2\ ,
\end{equation}
\begin{equation}
  \label{notation2}
  \overbar{\mathbf{x}}=\begin{pmatrix}\overbar x^{(1)}\\\overbar x^{(2)}\end{pmatrix}\text{ with }\overbar x^{(1)},\overbar x^{(2)}\in\calm_2\ ,
\end{equation}
\begin{equation}
  \label{notation3}
  \overbar{\mathbf{x}}=\left(\xphi{x^{(1)}_0}{x^{(2)}_0},\xphi{x^{(1)}_1}{x^{(2)}_1}\right)\text{ with }\xphi{x^{(1)}_0}{x^{(2)}_0}\in\R^2,\ \xphi{x^{(1)}_1}{x^{(2)}_1}\in L^2(-d,0;\R^2)\ ,
\end{equation}
\begin{equation}
  \label{notation4}
  \overbar{\mathbf{x}}=\left(\mathbf{x_0},\mathbf{x_1}\right)\text {with }\mathbf{x_0}\in\R^2,\ \mathbf{x_1}\in L^2(-d,0;\R^2)\ ;
\end{equation}
We then rewrite system \eqref{system1} in a more compact form. First for any fixed $\calf_{t_0}$-measurable $\calm^2_2$-valued random variable $\overbar{\mathbf{m}}=\left(\mathbf{m_0},\mathbf{m_1}\right)$ we consider the $2$-dimensional system
\begin{align}
\label{system2}
\begin{cases}
{\rm d}\mathbf{n}(t_0;t)=\left[C_0\mathbf{n}(t_0;t)+ \int_{-d}^0 \phi(s)\mathbf{n}(t_0;t+s)\,{\rm d}s\right]{\rm d}t
\\
\mathbf{n}(t_0;t_0)=\mathbf{m_0},
\\
\mathbf{n}(t_0;t_0+s)=\mathbf{m_1}(s), \qquad s \in[-d,0].
\end{cases}
\end{align}
where
\begin{equation*}
C_0:=\left(
\begin{array}{cc}
\epsilon+\my-\sigma_y\cdob \kappa& -\epsilon\\
0& \my
\end{array}
\right)
.
\end{equation*}
The following is a simple generalization of \cite[Part II, Chapter 4, Theorem 3.2]{BENSOUSSAN_DAPRATO_DELFOUR_MITTER} to random initial conditions.
\begin{lemma}
  Given any fixed $\calf_{t_0}$-measurable $\calm^2_2$-valued random variable $\overbar{\mathbf{m}}$, the Cauchy problem (\ref{system2}) has a unique absolutely continuous solution. Moreover system \eqref{system1} is equivalent to \eqref{system2} above when we choose
\begin{equation}
  \label{vector_conditions}
  \mathbf{m_0}=\begin{pmatrix}y(t_0)\\\E[y(t_0)]\end{pmatrix}, \mathbf{m_1}=\begin{pmatrix}y(t_0+\cdot)\\\E[y(t_0+\cdot)]\end{pmatrix}\ ;
\end{equation}
indeed in this case $\mathbf{n}(t_0;t)=\xphi{M_{t_0}}{e}(t)$ for every $t\in[t_0-d,+\infty)$.
\end{lemma}
We define next the operator $A_0:\mathcal{D}(A_0) \subset \mathcal{M}_2^2 \rightarrow \mathcal{M}_2^2$ as
\begin{equation}
  \begin{gathered}
  \label{defA0}
\mathcal{D}(A_0):= \left\{ \left(\mathbf{x_0},\mathbf{x}_1\right) \in \mathcal{M}^2_2: \mathbf{x_1} \in W^{1,2}(-d,0;\mathbb{R}^2), \ \mathbf{x_1}(0)=\mathbf{x_0}\right\},\\
A_0\left(\mathbf{x_0},\mathbf{x_1}\right):= \left(C_0\mathbf{x_0}+ \int_{-d}^0 \phi(s)\mathbf{x_1}(s)\,{\rm d}s,\frac{{\rm d}}{{\rm d}s}\mathbf{x_1}\right).
  \end{gathered}
\end{equation}

We can then reformulate system \eqref{system2} above as an evolution equation in $\calm^2_2$. Consider, again for any fixed $\calf_{t_0}$-measurable $\calm^2_2$-valued random variable $\overbar{\mathbf{m}}=\left(\mathbf{m_0},\mathbf{m_1}\right)$, the $\calm^2_2$-valued process $\overbar{\mathbf{N}}(t_0;\cdot)$ that is the solution on $[t_0,+\infty)$ of
\begin{equation}\label{INFINITE_DIMENSIONAL_STATE_EQUATION}
\begin{cases}
{\rm d} \overbar{\mathbf{N}}(t_0;t) = A_0 \overbar{\mathbf{N}}(t_0;t) {\rm d}t,\\
\overbar{\mathbf{N}}(t_0;t_0) =\overbar{\mathbf{m}}\ .
\end{cases}
\end{equation}
We collect now some useful results about the above equation, also for later reference; definitions of strict and weak solutions can be found for example in \cite[Appendix A]{DAPRATO_ZABCZYK_RED_BOOK}. Proofs are given in the Appendix.
\begin{proposition}
  \label{prop_semigroup}
  \begin{enumerate}[label=$(\roman{*})$]
  \item\label{item:1} The operator $A_0$ generates a strongly continuous semigroup $\left\{S(t)\right\}_{t\geq 0}$ in $\mathcal{M}^2_2$.
  \item $S(t)$ is a compact operator for every $t\geq d$.
  \item For every $\calf_{t_0}$-measurable $\calm^2_2$-valued random variable $\overbar{\mathbf{m}}\in\calm_2^2$ the process
\begin{equation}
  \label{semigroup_1}
S(t-t_0)\overbar{\mathbf{m}}\ ;
\end{equation}
is the unique weak (in distributional sense) solution of (\ref{INFINITE_DIMENSIONAL_STATE_EQUATION}); in particular
\begin{equation}
  \label{semigroup_2}
  \overbar{\mathbf{N}}(t_0;t)=\overbar{\mathbf{N}}(0;t-t_0)\ .
\end{equation}
Moreover if $\overbar{\mathbf{m}}\in\cald(A_0)$ a.s. then the solution is actually strict.

\item The Cauchy problem (\ref{INFINITE_DIMENSIONAL_STATE_EQUATION}) is equivalent to (\ref{system2}).
\item\label{item:5} Let $y$ be a solution of the second equation in (\ref{DYNAMICS_WEALTH_LABOR_INCOME}) on $[0,t_0]$; when choosing $\overbar{\mathbf{m}}$ as in (\ref{vector_conditions}), (\ref{INFINITE_DIMENSIONAL_STATE_EQUATION}) is equivalent to (\ref{system1}) and in this case we have
\begin{equation*}
  \overbar{\mathbf{N}}(t_0;t)=S(t-t_0)\overbar{\mathbf{m}}=\left(\mathbf{n}(t_0;t),\mathbf{n}(t_0;t+\cdot)\right)=\left(\begin{pmatrix}M_{t_0}(t)\\e(t)\end{pmatrix},\begin{pmatrix}\left\{M_{t_0}(t+s)\right\}_{s\in[-d,0]}\\\left\{e(t+s)\right\}_{s\in[-d,0]}\end{pmatrix}\right)\ .
\end{equation*}
\end{enumerate}
\end{proposition}

\subsection{Spectral properties of $A_0$}
\label{subsec:spectral}
The following result is an immediate consequence of \cite[Chapter 7, Lemma 2.1 and Theorem 4.2]{HALE_VERDUYN_LUNEL_BOOK}
\begin{lemma}
  \label{lemma_speck}
\begin{enumerate}[label=$(\roman{*})$]
\item  The spectrum of the operator $A_0$ is given by
\begin{equation*}
\{\lambda \in \mathbb{C}:K(\lambda)=0\},
\end{equation*}
where
\begin{equation*}
K(\lambda):= K_1(\lambda) K_2(\lambda),
\end{equation*}
with
\begin{equation*}
K_1(\lambda):= \lambda- (\epsilon+\my-\sigma_y\cdob \kappa)-\int_{-d}^0 e^{\lambda s}\phi(s)\, {\rm d}s,
\end{equation*}
\begin{equation*}
K_2(\lambda):= \lambda-\my- \int_{-d}^0 e^{\lambda s}\phi(s)\, {\rm d}s=K_1(\lambda)-(\epsilon-\sigma_y\cdob\kappa)\ .
\end{equation*}
In particular the spectral bound of $A_0$ is
\begin{equation*}
  \lambda_0=\sup\left\{{\rm Re}\lambda \colon K(\lambda)=0\right\}\ .
\end{equation*}
\item The spectrum of $A_0$ coincides with its point spectrum and is a discrete (thus countable) set.
\end{enumerate}
\end{lemma}
We can explicitly compute the resolvent operator of $A_0$ (a proof is sketched in the Appendix):
\begin{lemma}
\label{lemma_resolvent}
  Let $R(A_0)$ denote the resolvent set of $A_0$ and let $\lambda \in \mathbb{R} \cap R(A_0)$; then the resolvent operator of $A_0$ at $\lambda$ is given by
\begin{equation*}
R(\lambda,A_0)\left(\xphi{m_0}{e_0},\xphi{m_1}{e_1}\right)=\left(\xphi{u_0}{v_0},\xphi{u_1}{v_1}\right),
\end{equation*}
with
\begin{align}
\nonumber v_0&=\frac{1}{K_2(\lambda)}\left[e_0+ \int_{-d}^0\int_{-d}^s e^{-\lambda(s-\tau)}\phi(\tau)\,{\rm d}\tau \ e_1(s)\,{\rm d}s \right]\ ,\\
\label{u0}
u_0&=\frac{1}{K_1(\lambda)}\left[m_0+ \int_{-d}^0\int_{-d}^s e^{-\lambda(s-\tau)}\phi(\tau)\,{\rm d}\tau\ m_1(s)\,{\rm d}s\right]
\\
\nonumber&\phantom{=}-\frac{\epsilon}{K(\lambda)}\left[e_0+ \int_{-d}^0\int_{-d}^s e^{-\lambda(s-\tau)}\phi(\tau)\,{\rm d}\tau\ e_1(s)\,{\rm d}s\right]\ ,\\
\nonumber u_1(s)&=u_0e^{\lambda s}+\int_s^0e^{-\lambda(\tau-s)}m_1(\tau){\rm d} \tau\ ,\\
\nonumber v_1(s)&=v_0e^{\lambda s}+\int_s^0e^{-\lambda(\tau-s)}e_1(\tau){\rm d} \tau\ .
\end{align}
\end{lemma}
A crucial tool will be the following well known fact:
\begin{lemma}
  \label{lem_laplace}
 For every real $\lambda$ such that $\lambda>\lambda_0$ and every $\overbar{\mathbf{m}}\in\calm^2_2$ we have
    \begin{equation}
      \label{LAPLACE_SEMIGROUP}
  \int_0^{+\infty}e^{-\lambda t}S(t)\overbar{\mathbf{m}}{\rm d}t=R(\lambda,A_0)\overbar{\mathbf{m}}\ .
\end{equation}
\end{lemma}
\begin{proof}
 Identity (\ref{LAPLACE_SEMIGROUP}) is well known to hold for all real $\lambda$ that are larger than the type of $S(t)$. Since $S(t)$ is compact for every $t\geq d$, its type is actually equal to its spectral radius $\lambda_0$, see for example \cite[Part II, Chapter 1, Corollary 2.5]{BENSOUSSAN_DAPRATO_DELFOUR_MITTER}.
\end{proof}
Note asking $\lambda>\lambda_0$ is more restrictive than just asking that $\lambda\in \R\cap R(A_0)$.

\subsection{A formula for the human capital}
\label{subsec:formula_HC}
\begin{theorem}
\label{thm_HC}
  Assume that $r+\delta>\lambda_0$; then
  \begin{multline*}
    HC(t_0)=\frac{1}{K_1(r+\delta)}\left[y(t_0)+ \int_{-d}^0\int_{-d}^s e^{-(r+\delta)(s-\tau)}\phi(\tau)\,{\rm d}\tau \ y(t_0+s)\,{\rm d}s\right]\\
-\frac{\epsilon}{K(r+\delta)}\left[e(t_0)+ \int_{-d}^0\int_{-d}^s e^{-(r+\delta)(s-\tau)}\phi(\tau)\,{\rm d}\tau\  e(t_0+s)\,{\rm d}s\right]\ .
  \end{multline*}
\end{theorem}
\begin{proof}
Recall (\ref{def_HC}) and let $\overbar{\mathbf{m}}=\left(\mathbf{m_0},\mathbf{m_1}\right)$ be given by (\ref{vector_conditions}). We can rewrite the Human Capital as follows; we denote here by $P_{1,0}$ the projection on the first finite dimensional component of $\calm_2^2$, i.e.
\begin{equation*}
  P_{1,0}\begin{pmatrix}(x^{(1)}_0,x^{(1)}_1)\\(x^{(2)}_0,x^{(2)}_1)\end{pmatrix}=x^{(1)}_0\ .
\end{equation*}
We have
\begin{align*}
\frac{1}{\xi(t_0)}\mathbb{E}& \left[\int_{t_0}^{\infty}\xi(s)y(s){\rm d}s|\mathcal{F}_{t_0}\right]=e^{(r+\delta)t_0}\int_{t_0}^{\infty}e^{-(r+\delta)s}\tilde{\mathbb{E}}\left[y(s)|\mathcal{F}_{t_0}\right]\, {\rm d}s&\text{(by (\ref{EXPRESSION_II}))}\\
  &=e^{(r+\delta)t_0}\int_{t_0}^{\infty}e^{-(r+\delta)s}M_{t_0}(s)\, {\rm d}s& \\
  &=e^{(r+\delta)t_0}\int_{t_0}^{\infty}e^{-(r+\delta)s}P_{1,0}\left[\overbar{\mathbf{N}}(t_0;s)\right]\, {\rm d}s&\text{(by Prop. \ref{prop_semigroup}~-~\ref{item:5})}\\
  &=e^{(r+\delta)t_0}\int_{0}^{\infty}e^{-(r+\delta)t_0}e^{-(r+\delta)s}P_{1,0}\left[\overbar{\mathbf{N}}(0;s)\right]\, {\rm d}s&\text{(by (\ref{semigroup_2}))}\\
  &= \int_0^{\infty}e^{-(r+\delta)s}P_{1,0}\left[S(s)\overbar{\mathbf{m}}\right]\,{\rm d}s&\text{(by (\ref{semigroup_1}))}\\
  &=P_{1,0}\left[R(r+\delta,A_0)\overbar{\mathbf{m}}\right]&\text{(by Lemma \ref{lem_laplace})}\\
  &=\frac{1}{K_1(r+\delta)}\left[y(t_0)+ \int_{-d}^0\int_{-d}^s e^{-(r+\delta)(s-\tau)}\phi(\tau)\,{\rm d}\tau \ y(t_0+s)\,{\rm d}s\right]&\\
  &\phantom{=}-\frac{\epsilon}{K(r+\delta)}\left[e(t_0)+ \int_{-d}^0\int_{-d}^s e^{-(r+\delta)(s-\tau)}\phi(\tau)\,{\rm d}\tau\  e(t_0+s)\,{\rm d}s\right]&\text{(by Lemma \ref{lemma_resolvent}).}
\end{align*}
\end{proof}
Defining now
\begin{gather}
  K_1:=K_1(r+\delta)= r+\delta- (\epsilon+\my-\sigma_y\cdob \kappa)-\int_{-d}^0 e^{(r+\delta) s}\phi(s)\, {\rm d}s,  \ g_\infty=\frac{1}{K_1},
  \nonumber
  \\
  \label{defgiinfty}
  K_2:=K_2(r+\delta)=K_1(r+\delta)+(\epsilon-\sigma_y\cdob\kappa),\ i_{\infty}=\frac{\epsilon}{K_2},
\end{gather}
\begin{equation}\label{defhinfty}
G(s):=\int_{-d}^s e^{-(r+\delta)(s-\tau)}\phi(\tau){\rm d}\tau \text{ and } \ h_\infty(s):=g_\infty G(s)
\end{equation}
we finally obtain that, if $r+\delta>\lambda_0$,
\begin{align}
\label{HC}
HC(t_0)
&= g_\infty\left[y(t_0)+ \int_{-d}^0 G(s)y(t_0+s)\, {\rm d}s\right]
          - g_\infty i_\infty\left[ e(t_0)+\int_{-d}^0 G(s)e(t_0+s)\, {\rm d}s\right]\\
  &=\langle (g_\infty,h_\infty),(y(t_0),y(t_0+\cdot))\rangle_{\calm_2}-i_\infty\langle (g_\infty,h_\infty),(e(t_0),e(t_0+\cdot))\rangle_{\calm_2}\ .
\end{align}

The above representation allows to rewrite the constraint (\ref{NO_BORROWING_WITHOUT_REPAYMENT_CONDITIONLA_MEAN}) as
\begin{equation}
  \label{constraint_rewritten}
  W(t)+\langle (g_\infty,h_\infty),(y(t),y(t+\cdot))\rangle_{\calm_2}-i_\infty\langle (g_\infty,h_\infty),(e(t),e(t+\cdot))\rangle_{\calm_2}\geq 0,
\end{equation}
which in turn suggests to set
\begin{equation*}
  \overbar{l}_\infty=(g_\infty,h_\infty)\in\calm_2,\quad \overbar{\mathbf{l}}_\infty=\xphi{\overbar{l}_\infty}{-i_\infty\overbar{l}_\infty}\in\calm_2^2
\end{equation*}
and to define the function $\R\times\calm_2^2\to\R$
\begin{align}
  \label{gammainfinity}\Gamma_\infty(w,\overbar{\mathbf{x}}):&=w+\langle \overbar{\mathbf{l}}_\infty,\overbar{\mathbf{x}}\rangle_{\calm_2^2}\\
\nonumber                                                             &=w+\langle \overbar{l}_\infty,\overbar{x}^{(1)}\rangle_{\calm_2}-i_\infty\langle \overbar{l}_\infty,\overbar{x}^{(2)}\rangle_{\calm_2}
\\
\nonumber &=w+g_\infty x^{(1)}_0+\langle h_\infty,x^{(1)}_1\rangle-i_\infty \left[g_\infty x^{(2)}_0+\langle h_\infty,x^{(2)}_1\rangle\right]
\end{align}
for $w\in\R$, $\overbar{\mathbf{x}}\in\calm^2$, and to consider the sets
\begin{equation}
  \label{constraint_sets}
  \begin{gathered}
  \calh:=\R\times\calm_2^2\ ,\\
  \calh_+=\left\{(w,\overbar{\mathbf{x}})\in\calh\colon \Gamma_\infty(w,\overbar{\mathbf{x}})\geq 0\right\}\ ,\\
  \calh_{++}=\left\{(w,\overbar{\mathbf{x}})\in\calh\colon\Gamma_\infty(w,\overbar{\mathbf{x}})>0\right\}\ .
\end{gathered}
\end{equation}
$\calh$ is naturally a Hilbert space when endowed with the inner product
\begin{equation*}
  \left\langle \left(a,\overbar{\mathbf{x}}\right),\left(b,\overbar{\mathbf{y}}\right)\right\rangle_\calh=\left\langle \left(a,\overbar{x}^{(1)},\overbar{x}^{(2)}\right),\left(b,\overbar{y}^{(1)},\overbar{y}^{(2)}\right)\right\rangle_{\calh}:=ab+\left\langle \overbar{x}^{(1)},\overbar{y}^{(1)}\right\rangle_{\calm_2}+\left\langle \overbar{x}^{(2)},\overbar{y}^{(2)}\right\rangle_{\calm_2}\ .
\end{equation*}

\subsection{Why Assumption \ref{Hyp_K}}
\label{subsec:comment}
The requirement
\begin{equation*}
  r+\delta>\lambda_0
\end{equation*}
in Theorem \ref{thm_HC} is difficult to check in practice, as it requires an explicit computation of $\lambda_0$. Therefore we look for some sufficient condition, possibly easier to check, for such requirement to be satisfied. Set for $\lambda\in\C$
\begin{equation*}
\widetilde{K}_1(\lambda):= \lambda- (\epsilon+\my-\sigma_y\cdob \kappa)-\int_{-d}^0 e^{\lambda s}\vert\phi(s)\vert\, {\rm d}s,
\end{equation*}
\begin{equation*}
\widetilde{K}_2(\lambda):= \lambda-\my- \int_{-d}^0 e^{\lambda s}\vert\phi(s)\vert\, {\rm d}s=\widetilde{K}_1(\lambda)-(\epsilon-\sigma_y\cdob\kappa)\ ,
\end{equation*}
\begin{equation*}
\widetilde{K}(\lambda):= \widetilde{K}_1(\lambda) \widetilde{K}_2(\lambda)\ .
\end{equation*}
Finally set
\begin{equation*}
  \tilde\lambda_0=\sup\left\{\Re(\lambda)\colon\widetilde{K}(\lambda)=0\right\}\ .
\end{equation*}
Note that $\tilde\lambda_0$ is the spectral radius of the operator $\widetilde{A}_0:\mathcal{D}(\widetilde{A}_0) \subset \mathcal{M}_2^2 \rightarrow \mathcal{M}_2^2$,
\begin{equation*}
  \begin{gathered}
\mathcal{D}(\widetilde{A}_0):= \left\{ \left(\mathbf{x_0},\mathbf{x}_1\right) \in \mathcal{M}^2_2: \mathbf{x_1} \in W^{1,2}(-d,0;\mathbb{R}^2), \ \mathbf{x_1}(0)=\mathbf{x_0}\right\},\\
\widetilde{A}_0\left(\mathbf{x_0},\mathbf{x_1}\right):= \left(C_0\mathbf{x_0}+ \int_{-d}^0 \vert\phi(s)\vert\mathbf{x_1}(s)\,{\rm d}s,\frac{{\rm d}}{{\rm d}s}\mathbf{x_1}\right)\ .
  \end{gathered}
\end{equation*}

\begin{lemma}
  \label{lemma2}
  The functions $\widetilde{K}_1$ and $\widetilde{K}_2$, restricted to the real numbers, are strictly increasing.\\
Moreover the function $\widetilde{K}$ restricted to the real numbers is continuous and such that:
\begin{enumerate}[label=$(\roman{*})$]
  \item $\lim_{\xi\to\pm\infty}\widetilde{K}(\xi)=+\infty$;
\item the equation $\widetilde{K}(\xi)=0$ admits exactly two real solutions $\xi_1$ and $\xi_2$ with $\widetilde{K}_1(\xi_1)=0$ and $\widetilde{K}_2(\xi_2)=0$, and $\xi_1=\xi_2$ if and only if $\epsilon-\sigma_y\cdob\kappa=0$;
\item setting $\xi_0 :=\max(\xi_1, \xi_2)$ we have
\begin{equation*}
\xi_0=
 \begin{cases}
 \xi_2 \qquad  \text{if} \quad \epsilon-\sigma_y\cdob \kappa>0,
 \\
  \xi_1 \qquad  \text{if} \quad \epsilon-\sigma_y\cdob \kappa<0
   \end{cases}
\end{equation*}
and eventually it holds
$\tilde\lambda_0=\xi_0$.
\end{enumerate}
\end{lemma}
\begin{proof}
By definition $\widetilde{K}(\xi)=0$ if either $\widetilde{K}_1(\xi)=0$ or $\widetilde{K}_2(\xi)=0$. It is immediate to check that both $\widetilde{K}_1$ and $\widetilde{K}_2$ are continuous strictly increasing functions on $\R$ with $\lim_{\xi\to\pm\infty}\widetilde{K}_j(\xi)=\pm\infty$, $j=1,2$. Therefore $\widetilde{K}$ is continuous on $\R$ and there exists exactly one value $\xi_1\in\R$ such that $\widetilde{K}_1(\xi_1)=0$ and exactly one value $\xi_2\in\R$ such that $\widetilde{K}_2(\xi_2)=0$. Since $\widetilde{K}_2(\xi)=\widetilde{K}_1(\xi)-(\epsilon-\sigma_y\cdob\kappa)$, $\xi_1=\xi_2$ if and only if $\epsilon-\sigma_y\cdob\kappa=0$ and
\begin{equation*}
\max(\xi_1, \xi_2)=
 \begin{cases}
 \xi_2 \qquad  \text{if} \quad \epsilon-\sigma_y\cdob \kappa>0,
 \\
  \xi_1 \qquad  \text{if} \quad \epsilon-\sigma_y\cdob \kappa<0.
   \end{cases}
\end{equation*}
Let us now prove that $\tilde\lambda_0=\xi_0$. By definition we have $\xi_0 \le \tilde\lambda_0$. In order to prove that $\xi_0\ge\tilde\lambda_0$, let us consider an arbitrary $\lambda =a + i b \in\mathbb{C}$ such that $\widetilde{K}(\lambda)=0$. Suppose first that $\widetilde{K}_1(\lambda)=0$, so that in particular
\begin{align*}
  0&=\Re\left(\widetilde{K}_1(\lambda)\right)=a-(\epsilon+\my-\sigma_y\cdob\kappa)-\int_{-d}^0e^{as}\cos(bs)\vert\phi(s)\vert{\rm d}s\\
   &\geq a-(\epsilon+\my-\sigma_y\cdob\kappa)-\int_{-d}^0e^{as}\vert\phi(s)\vert{\rm d}s =\widetilde{K}_1(a)\ .
\end{align*}
Therefore $a\leq\xi_1\leq\xi_0$. If instead $\widetilde{K}_2(\lambda)=0$ then
\begin{equation*}
  0=\Re\left(\widetilde{K}_2(\lambda)\right)=a-\my-\int_{-d}^0e^{as}\cos(bs)\vert\phi(s)\vert{\rm d}s\geq a-\my-\int_{-d}^0e^{as}\vert\phi(s)\vert{\rm d}s =\widetilde{K}_2(a)\
\end{equation*}
hence $a\leq\xi_2\leq\xi_0$. In both cases taking the supremum in the definition of $\tilde\lambda_0$ yields $\tilde\lambda_0\leq\xi_0$.
\end{proof}
The convenience of introducing $\tilde\lambda_0$ is clarified by its relation with $\lambda_0$.
\begin{lemma}
  We have
  \begin{equation*}
    \lambda_0\leq\tilde\lambda_0\ .
  \end{equation*}
\end{lemma}
\begin{proof}
  For every $\lambda=a+ib\in\C$ we have
  \begin{equation}
    \label{eq:kappa1}
    \Re\left(K_1(\lambda)\right)=a-(\epsilon+\my-\sigma_y\cdob\kappa)-\int_{-d}^0e^{as}\cos(bs)\phi(s){\rm d} s\ ,
  \end{equation}
  \begin{equation}
    \label{eq:kappa2}
    \Re\left(K_2(\lambda)\right)=a-\my-\int_{-d}^0e^{as}\cos(bs)\phi(s){\rm d} s\ .
  \end{equation}
Suppose first that $\xi_0=\xi_1$. Then $\epsilon-\sigma_y\cdob\kappa<0$ and $\widetilde{K}_1(\xi)<\widetilde{K}_2(\xi)$ for every $\xi\in\R$. Therefore, recalling the definition of $\lambda_0$, it is enough to show that for every number $\lambda=a+i b$ such that $K(\lambda)=0$ we have $\widetilde{K}_2(a)\leq 0$. So let $\lambda$ be such a number; we have
\begin{align*}
  \widetilde{K}_2(a)&=a-\my-\int_{-d}^0e^{as}\vert\phi(s)\vert{\rm d}s\\
                    &=a-\my-\int_{-d}^0e^{as}\cos(bs)\phi(s){\rm d}s-(\epsilon-\sigma_t\cdob\kappa)+(\epsilon-\sigma_y\cdob\kappa)\\
                    &\phantom{a-\my-\int_{-d}^0e^{as}\cos(bs)\phi(s){\rm d}s-}+\int_{-d}^0e^{as}\left(\cos(bs)\phi(s)-\vert\phi(s)\vert\right){\rm d}s\\
                    &\leq \int_{-d}^0e^{as}\left(\cos(bs)\phi(s)-\vert\phi(s)\vert\right){\rm d}s\\
                    &\leq 0\ ,
\end{align*}
where to deduce the second to last inequality one uses (\ref{eq:kappa1}) together with the fact that $\epsilon-\sigma_y\cdob\kappa<0$ if $K_1(\lambda)=0$ and (\ref{eq:kappa2}) if $K_2(\lambda)=0$.

Similarly if $\xi_0=\xi_2$ we have $\epsilon-\sigma_y\cdob\kappa\geq 0$ and is suffices to show that for any $\lambda=a+i b$ as above $\widetilde{K}_1(a)\leq 0$. In this case we find
\begin{align*}
  \widetilde{K}_1(a)&=a-\my-(\epsilon-\sigma_y\cdob\kappa)-\int_{-d}^0e^{as}\vert\phi(s)\vert{\rm d}s\\
                   &=a-\my-(\epsilon-\sigma_y\cdob\kappa)-\int_{-d}^0e^{as}\cos(bs)\phi(s){\rm d}s+\int_{-d}^0e^{as}\left(\cos(bs)\phi(s)-\vert\phi(s)\vert\right){\rm d}s\\
                   &\leq \int_{-d}^0e^{as}\left(\cos(bs)\phi(s)-\vert\phi(s)\vert\right){\rm d}s\\
                   &\leq 0
\end{align*}
and the conclusion follows as before.
\end{proof}

Thanks to this last result it becomes clearer why we make Assumption \ref{Hyp_K}. If $r$ and $\delta$ are such that Assumption \ref{Hyp_K} is satisfied, the in particular both $\widetilde{K}_1(r+\delta)$ and $\widetilde{K}_2(r+\delta)$ are positive and we have that
\begin{equation*}
  r+\delta>\xi_0=\tilde\lambda_0\geq\lambda_0;
\end{equation*}
the assumption of Theorem \ref{thm_HC} is therefore satisfied, so that the constraint (\ref{NO_BORROWING_WITHOUT_REPAYMENT_CONDITIONLA_MEAN}) takes the convenient formulation (\ref{constraint_rewritten}). Note that Assumption \ref{Hyp_K} is not equivalent to saying that $r+\delta>\lambda_0$ but only a sufficient condition. However it is usually much easier to verify the inequalities in Assumption \ref{Hyp_K} than computing explicitly $\lambda_0$.
\begin{remark}
  Repeating the arguments above one can easily show that if there is no mean-reverting effect (i.e. $\epsilon=0$) and $\phi$ is positive almost everywhere then $r+\delta>\lambda_0$ if and only if $K_1(r+\delta)>0$.
\end{remark}

\section{The general problem and the associated HJB equation}
\label{SEC:HJB}

\subsection{The infinite dimensional general  problem}
\label{sub:pbinfdim}

Here we rewrite, in a suitable infinite-dimensional setting,
the problem of maximizing the functional \eqref{OBJECTIVE_FUNCTION}
under the constraint \eqref{NO_BORROWING_WITHOUT_REPAYMENT_CONDITIONLA_MEAN}
and under the transformed state equation \eqref{DYNAMICS_WEALTH_LABOR_INCOMEe}.
We call it 'general  problem' since we consider it for generic initial data, so, not necessarily connected with our original control of McKean-Vlasov dynamics. In Section \ref{sec:back} we will go back to the original problem.

For any two Banach spaces $E$ and $E^\prime$ we will denote by $L(E;E^\prime)$ the space of bounded linear operators from $E$ to $E^\prime$.\\
We proceed similarly to the previous section to reformulate system (\ref{DYNAMICS_WEALTH_LABOR_INCOMEe}) in the infinite-dimensional space $\calm_2^2$. To begin with, we  define the finite-dimensional operator $C$ on $\R^2$ as
\begin{equation*}
C:=\left(
\begin{array}{cc}
\epsilon+\my& -\epsilon\\
0& \my
\end{array}
\right)
\end{equation*}
and the operator $A:\mathcal{D}(A) \subset \mathcal{M}_2^2 \rightarrow \mathcal{M}_2^2$ as
\begin{equation}
  \begin{gathered}
  \label{defA}
\mathcal{D}(A):= \left\{ \left(\mathbf{x_0},\mathbf{x_1}\right) \in \mathcal{M}^2_2: \mathbf{x_1} \in W^{1,2}(-d,0;\mathbb{R}^2), \ \mathbf{x_1}(0)=\mathbf{x_0}\right\},\\
A\left(\mathbf{x_0},\mathbf{x_1}\right):= \left(C\mathbf{x_0}+ \int_{-d}^0 \phi(s)\mathbf{x_1}(s)\,{\rm d}s,\frac{{\rm d}}{{\rm d}s}\mathbf{x_1}\right)
  \end{gathered}
\end{equation}
(the only difference between $A$ and $A_0$ defined in (\ref{defA0}) lies in the use of $C$ in place of $C_0$). By Proposition \ref{prop_semigroup}~-~\ref{item:1} the operator $A$ is the infinitesimal generator of a strongly continuous semigroup $T(t)$ in $\calm_2^2$.\\
Furthermore, we need to define the linear bounded operator $F\colon\calm_2^2\to L(\R^n;\calm_2^2)$ in the following way: for every $\overbar{\mathbf{x}}=\left(\xphi{x^{(1)}_0}{x^{(2)}_0},\xphi{x^{(1)}_1}{x^{(2)}_1}\right)\in\calm_2^2$, $F(\overbar{\mathbf{x}})$ is the linear map
\begin{equation}
  \label{defF}
  z\mapsto \left(P_{1,0}\overbar{\mathbf{x}}\right)\left(\begin{pmatrix}\sigma_y\cdob z\\0\end{pmatrix},\begin{pmatrix}0_{L^2}\\0_{L^2}\end{pmatrix}\right)=\left(\begin{pmatrix}x_0^{(1)}\sigma_y\cdob z\\0\end{pmatrix},\begin{pmatrix}0_{L^2}\\0_{L^2}\end{pmatrix}\right)\ ,
\end{equation}
where $0_{L^2}$ denotes the null function in $L^2(-d,0;\R)$.\\
Consider the equation in $\calm_2^2$
\begin{equation}
\label{Y_condensed}
  \begin{cases}
    {\rm d}\overbar{\mathbf{Y}}(t)=A\overbar{\mathbf{Y}}(t){\rm d}t+F\left(\overbar{\mathbf{Y}}(t)\right){\rm d} Z(t)\ ,\quad t\in[0,+\infty)\ ,\\
    \overbar{\mathbf{Y}}(0)=\overbar{\mathbf{x}}\ .
  \end{cases}
\end{equation}
Since only the first finite-dimensional component of $F\left(\overbar{\mathbf{Y}}(t)\right)$ is nonzero, the stochastic integral abowe is well-defined without the need to recurr to any infinite-dimensional stochastic integration theory. As paths of $\overbar{\mathbf{Y}}$ have almost surely the regularity of Brownian paths, $\overbar{\mathbf{Y}}(t)$ is almost surely not in $\cald(A)$ for every $t$; therefore as a solution of the above equation we mean a \emph{mild} solution, that is, almost surely $\overbar{\mathbf{Y}}(t)$ should satisfy
\begin{equation}
  \label{sol_semig}
  \overbar{\mathbf{Y}}(t)=T(t)\overbar{\mathbf{x}}+\int_0^tT(t-s)F\left(\overbar{\mathbf{Y}}(s)\right){\rm d} Z(s)
\end{equation}
for every $t\in[0,+\infty)$.

\begin{proposition}
\label{equiv_stoch}
  For every initial condition $\overbar{\mathbf{x}}$ equation (\ref{Y_condensed}) has a unique mild solution, that is also a weak solution, given by (\ref{sol_semig}), with continuous trajectories almost surely. When
  \begin{equation*}
    \overbar{\mathbf{x}}=\xphi{\overbar{x}}{\overbar{x}}=\left(\xphi{x_0}{x_0},\xphi{x_1}{x_1}\right)
  \end{equation*}
and $x_0$, $x_1$ are as in (\ref{DYNAMICS_WEALTH_LABOR_INCOME}), equation (\ref{Y_condensed}) and the equation for $y$ in (\ref{DYNAMICS_WEALTH_LABOR_INCOME}) are equivalent, meaning that if $y$ solves the second equation in (\ref{DYNAMICS_WEALTH_LABOR_INCOME}) then
\begin{equation}
  \label{eq:Yy}
  \overbar{\mathbf{Y}}(t)=\left(\begin{pmatrix}y(t)\\\mathbb{E}[y(t)]\end{pmatrix},\begin{pmatrix}\left\{y(t+s)\right\}_{s\in[-d,0]}\\\left\{\mathbb{E}[y(t+s)]\right\}_{s\in[-d,0]}\end{pmatrix}\right)\ ;
\end{equation}
is a mild solution of (\ref{Y_condensed}), and conversely if $\overbar{\mathbf{Y}}$ is a solution of (\ref{Y_condensed}) then $P_{1,0}\overbar{\mathbf{Y}}$ is a solution of the second equation in (\ref{DYNAMICS_WEALTH_LABOR_INCOME}). In particular there exists a unique (probabilistically) strong solution $y$ of the second equation in (\ref{DYNAMICS_WEALTH_LABOR_INCOME}) with initial condition $\overbar{x}=(x_0,x_1)$.
\end{proposition}

\begin{proof}
  Existence and uniqueness of a continuous mild solution of (\ref{Y_condensed}) and the equivalence between mild and weak solutions follow from \cite[Theorem 6.7]{DAPRATO_ZABCZYK_RED_BOOK}. The equivalence property has been proven in \cite[Theorem 3.1]{CHOJNOWSKA-MICHALIK_1978}.
\end{proof}
\begin{remark}
  In the setting of the above proposition the map $[-d,0]\ni s\mapsto y(t+s)$ may fail to be continuous for small times, since the initial condition is only in $L^2$. However this cannot happen if the initial condition $x_1$ is continuous and $x_0=x_1(0)$. In this case everything can be formulated in spaces of continuous functions, see for example \cite[Section 5]{FLANDOLI_RUSSO_ZANCO} for a discussion.
\end{remark}

Using (\ref{Y_condensed}), for any given $(\theta,c,B)\in\Pi^0$ and $(w,\overbar{\mathbf{x}})\in\calh$ system (\ref{DYNAMICS_WEALTH_LABOR_INCOME}) can be formulated in the unknown $(W,\overbar{\mathbf{Y}})\in\calh=\R\times\calm_2^2$ as
\begin{equation}
  \label{systemH}
  \begin{cases}
    {\rm d} W(t)=\left[(r+\delta)W(t)+\theta\cdob(\mu-r\mathbf{1})-c(t)-\delta B(t)+P_{1,0}\left(\overbar{\mathbf{Y}}(t)\right)\right]{\rm d}t+\theta(t)\cdob\sigma{\rm d}Z(t)\ ,\\
    {\rm d} \overbar{\mathbf{Y}}(t)=A\overbar{\mathbf{Y}}(t){\rm d}t+F\left(\overbar{\mathbf{Y}}(t)\right){\rm d}Z(t)\ ,\\
    \left(W(0),\overbar{\mathbf{Y}}(0)\right)=\left(w,\overbar{\mathbf{x}}\right)\
  \end{cases}
\end{equation}
and by Proposition \ref{equiv_stoch} it  admits a unique strong solution $(W,y)$ corresponding to the unique solution $(W,\overbar{\mathbf{Y}})$ of the latter.\\
We will introduce in a moment a convenient formulation of our optimal control problem and the Hamilton-Jacobi-Bellman equation associated to it. To this end we condense further the notation defining two linear operators on $\calh\times\R^n\times\R_+\times\R_+$: $\calb$ is the unbounded operator with values in $\calh$ given by
\begin{equation*}
  \begin{gathered}
    D(\calb)=\left(\R\times D(A)\right)\times\R^n\times\R_+\times\R_+\ ,\\
    \calb\left(w,\overbar{\mathbf{x}},\theta,c,B\right)=\left((r+\delta)w+\theta\cdob(\mu-r\mathbf{1})-c-\delta B+P_{1,0}\overbar{\mathbf{x}},A\overbar{\mathbf{x}}\right)\ ,
  \end{gathered}
\end{equation*}
while $\cals$ is the bounded operator with values in $L(\R^n;\calh)$
\begin{equation*}
    \cals\left(w,\overbar{\mathbf{x}},\theta,c,B\right)=\left[z\mapsto \left(\sigma^\top\theta\cdob z,F(\overbar{\mathbf{x}})z\right)\right]\ .
\end{equation*}
For simplicity we let both of them depend also on the variables in $\calh\times\R^n\times\R_+\times\R_+$ that do not explicitly appear in their definition; in particular $\cals$ depends only on $x^{(1)}_0$ and $\theta$.\\
It is not difficult to check that for every fixed $\left(w,\overbar{\mathbf{x}},\theta,c,B\right)\in\calh\times\R^n\times\R_+\times\R_+$ the adjoint of $\cals\left(w,\overbar{\mathbf{x}},\theta,c,B\right)$ is the map $\cals\left(w,\overbar{\mathbf{x}},\theta,c,B\right)^\ast\in L(\calh;\R^n)$ given by
\begin{equation*}
  (u,\overbar{\mathbf{p}})\mapsto u\sigma^\top\theta+P_{1,0}(\overbar{\mathbf{x}})P_{1,0}(\overbar{\mathbf{p}})\sigma_y\ .
\end{equation*}
Set \begin{align*}
  \cali:=L^2(-d,0;\R)&\times \Big(L^2(-d,0;\R)\times L\left(L^2(-d,0;\R);L^2(-d,0;\R)\right)\Big)\times\\
  &\times \Big(L^2(-d,0;\R) \times L\left(L^2(-d,0;\R);L^2(-d,0;\R)\right)\Big)\ ;
\end{align*}
then $L(\calh;\calh)\cong\caln:=\calh\times(\calh\times\cali)\times(\calh\times\cali)$ and given an element
\begin{equation*}
Q=\left(H_0,\begin{matrix}(H_1,I_1)\\(H_2,I_2)\end{matrix}\right)\in\caln
\end{equation*}
we can index its entries as $Q_{11},Q_{12},\dots,Q_{21},\dots,Q_{51},\dots,Q_{55}$; here $Q_{11},\dots,Q_{15}$ are the elements of $H_0$, in the order given by the definition of the space $\calh$, and so on. Through this interpretation we can define the space of symmetric elements in $\caln$ as
\begin{equation*}
  \caln_{\text{sym}}:=\left\{Q\in\caln\colon Q_{ij}=Q_{ji},i,j=1,\dots,5\right\}\ .
\end{equation*}
By simple computations we then have, for any $Q\in\caln_{\text{sym}}$ and any $\left(w,\overbar{\mathbf{x}},\theta,c,B\right)\in\calh\times\R^n\times\R_+\times\R_+$,
\begin{multline}
  \label{eq:trace}
\mathrm{Tr}\left[Q\cals\left(w,\overbar{\mathbf{x}},\theta,c,B\right)\cals\left(w,\overbar{\mathbf{x}},\theta,c,B\right)^\ast\right]\\=Q_{11}\left\vert\sigma^\top\theta\right\vert^2+Q_{22}P_{1,0}(\overbar{\mathbf{x}})^2\left\vert\sigma_y\right\vert^2+2Q_{12}P_{1,0}(\overbar{\mathbf{x}})\sigma_y\cdob\sigma^\top\theta\ .
\end{multline}
In particular for any given function
\begin{equation*}
  f\colon\calh\to\R
\end{equation*}
its second Fr\'echet derivative at a given point $(w,\overbar{\mathbf{x}})$ is an element $\nabla^2f(w,\overbar{\mathbf{x}})\in\caln_{\text{sym}}$, and the above formula provides the second order term that will appear in our Hamilton-Jacobi-Bellman equation.

Recall the set $\Pi^0$ defined in (\ref{DEF_PI0_FIRST_DEFINITION}). We denote a triple of controls $\left(\theta(\cdot),c(\cdot),B(\cdot)\right)\in\R^n\times\R_+\times\R_+$ as $\pi(\cdot)$; for given $\pi(\cdot)\in\Pi^0$ and given initial condition $\left(w,\overbar{\mathbf{x}}\right)$ we eventually rewrite system (\ref{systemH}) for the unknown $\calh$-valued process $\calx=\left(W,\overbar{\mathbf{Y}}\right)$ as
\begin{equation}
  \label{systemH_condensed}
  \begin{cases}
    {\rm d}\calx(t)=\calb\left(\calx(t),\pi(t)\right){\rm d}t+\cals\left(\calx(t),\pi(t)\right){\rm d}Z(t)\ ,\\
    \calx(0)=\left(w,\overbar{\mathbf{x}}\right)\ ,
  \end{cases}
\end{equation}
where again solutions are to be intended in mild sense, since almost surely $\calx\notin D(\calb)$. By \cite[Chapter 5.6]{KARATZAS_SHREVE_91} and Proposition \ref{equiv_stoch} there is a unique mild solution of (\ref{systemH_condensed}); we will denote such solution at time $t\geq 0$ as $\calx^{w,\overbar{\mathbf{x}}}(t;\pi)=\left(W^{w,\overbar{\mathbf{x}}}(t;\pi),\overbar{\mathbf{Y}}^{\overbar{\mathbf{x}}}(t)\right)$; recall that we are interested only in initial conditions of the form
\begin{equation}
  \label{innt_cond_sec_4}
  (w,\overbar{\mathbf{x}})=\left(w,\xphi{\overbar{x}}{\overbar{x}}\right)=\left(w,\xphi{x_0}{x_0},\xphi{x_1}{x_1}\right)\ .
\end{equation}
The dependence on the initial condition will be sometimes hidden if notationally convenient.\\
Thanks to the results proved in Section \ref{subsec:formula_HC} we can write the set of admissible controls as
\begin{align*}
  \Pi\left(w,\overbar{\mathbf{x}}\right)&=\big\{\pi\in\Pi^0\colon \calx^{w,\overbar{\mathbf{x}}}(t;\pi)\in\calh_+\ \forall t\geq 0\big\}\\
                                        &=\big\{\pi\in\Pi^0\colon \Gamma_\infty\left(W^{w,\overbar{\mathbf{x}}}(t;\pi),\overbar{\mathbf{Y}}^{\overbar{\mathbf{x}}}(t)\right)\geq 0\ \forall t\geq 0\big\}\\
                                        &=\big\{\pi\in\Pi^0\colon W^{w,\overbar{\mathbf{x}}}(t)+\langle (g_\infty,h_\infty),(y^{\overbar{x}}(t),y^{\overbar{x}}(t+\cdot))\rangle_{\calm_2}+\\
                                        &\phantom{aaaaaaaaaaaaaa}-i_\infty\langle (g_\infty,h_\infty),(\E[y^{\overbar{x}}(t)],\E[y^{\overbar{x}}(t+\cdot)])\rangle_{\calm_2}\geq 0\ \forall t\geq 0\big\}\\
                                        &=\big\{\pi\in\Pi^0\colon W^{w,\overbar{\mathbf{x}}}(t)+g_\infty y^{\overbar{x}}(t)+\langle h_\infty,y^{\overbar{x}}(t+\cdot)\rangle+\\
                                        &\phantom{aaaaaaaaaaaaaa}-i_\infty \big[g_\infty \E[y^{\overbar{x}}(t)]+\langle h_\infty,\E[y^{\overbar{x}}(t+\cdot)]\rangle\big]\geq 0\ \forall t\geq0\big\}\ .
\end{align*}
Recall the objective functional
\begin{equation}\tag{\ref{OBJECTIVE_FUNCTION}}
J(w,\overbar{\mathbf{x}};\pi):=\mathbb E \left(\int_{0}^{+\infty} e^{-(\rho+ \delta) t }
\left( \frac{c(t)^{1-\gamma}}{1-\gamma}
+ \delta \frac{\big(k B(t)\big)^{1-\gamma}}{1-\gamma}\right) {\rm d}t
\right)\ ;
\end{equation}
the part of the integrand in the definition of $J$ that depends on the controls is the \emph{utility function}
\begin{equation*}
  U(\pi)=\frac{c^{1-\gamma}}{1-\gamma}+ \delta \frac{\big(k B\big)^{1-\gamma}}{1-\gamma}\ .
\end{equation*}
Here we write for simplicity $J$ and $U$ as functions of $\pi$ although they actually only depend on the control triple through $c$ and $B$ and not through $\theta$.\\
Our goal is to solve the following:
\begin{problem}
\label{pbl}
  Under Assumptions \ref{hp:S}, \ref{hp:tau}, \ref{Hyp_K} and \ref{Hyp_gamma} and for given fixed $(w,\overbar{\mathbf{x}})\in\calh$ as in (\ref{innt_cond_sec_4}), find $\tilde{\pi}\in\Pi\left(w,\overbar{\mathbf{x}}\right)$ such that
  \begin{equation*}
    J\left(w,\overbar{\mathbf{x}};\tilde{\pi}\right)=\max_{\pi\in\Pi(w,\overbar{\mathbf{x}})}J\left(w,\overbar{\mathbf{x}};\pi\right)\ .
  \end{equation*}
\end{problem}

The following result can be proved with a straightforward adaptation of the proof of \cite[Proposition 3.1]{BGP}.
\begin{proposition}
  The adjoint operator of $A$ is the operator $A^\ast\colon\cald(A^\ast)\subset\calm_2^2\to\calm_2^2$ defined as
  \begin{equation}
    \label{Astar}
    \begin{gathered}
      \cald(A^\ast)\colon=\left\{\left(\mathbf{y_0},\mathbf{y_1}\right)\colon\mathbf{y_1}\in W^{1,2}\left([-d,0];\mathbb{R}^2\right),\mathbf{y_1}(-d)=0\right\}\ ,\\
      A^\ast\left(\mathbf{y_0},\mathbf{y_1}\right)=\left(C^\top\mathbf{y_0}+\mathbf{y_1}(0),\mathbf{y_0}\phi-\frac{{\rm d}}{{\rm d} s}\mathbf{y_1}\right)\
    \end{gathered}
  \end{equation}
  where $C^\top$ is the transpose of the matrix $C$.
\end{proposition}

\subsection{The HJB equation}
\label{sub:HJB}

We now introduce the Hamiltonian for our control problem: formally we expect it to be the function
\begin{equation*}
  \widetilde\H\colon\big(\R\times D(A)\big)\times \calh\times \caln_{\text{sym}} \to \R\cup\{\pm \infty\}
\end{equation*}
given by
\begin{multline*}
  \widetilde\H\left(\left(w,\overbar{\mathbf{x}}\right),\left(u,\overbar{\mathbf{p}}\right),Q\right)\\
  =\sup_{\pi\in\R^n\times\R_+\times\R_+}\Big[\left\langle\calb\left(w,\overbar{\mathbf{x}},\pi\right),\left(u,\overbar{\mathbf{p}}\right)\right\rangle_\calh+\frac12\mathrm{Tr}\left[Q\cals\left(w,\overbar{\mathbf{x}},\pi\right)\cals\left(w,\overbar{\mathbf{x}},\pi\right)^\ast\right]+U(\pi)\Big]\ .
\end{multline*}
It is however convenient (see also Remark \ref{rem:H_tilde} below) to define it a bit differently. Using the definitions of $\calb$ and $\cals$ (and in particular the defining property of $A^\ast$ with respect to the duality product appearing in $\calb$) together with (\ref{eq:trace}) we can write the Hamiltonian in a more explicit way, separating at the same time the part that depends on the controls from the rest; we thus choose as the Hamiltonian for our problem the function
\begin{equation*}
  \H\colon\calh\times \big(\R\times D(A^\ast)\big)\times \caln_{\text{sym}} \to \R\cup\{\pm \infty\}
\end{equation*}
given by
\begin{equation*}
  \H\left(\left(w,\overbar{\mathbf{x}}\right),\left(u,\overbar{\mathbf{p}}\right),Q\right):=\H_0\left((w,\overbar{\mathbf{x}},(u,\overbar{\mathbf{p}}),Q_{22}\right)+\H_{\text{max}}\left(P_{1,0}\overbar{\mathbf{x}},u,Q_{11},Q_{12}\right)\ ,
\end{equation*}
where
\begin{align*}
  \H_0\left((w,\overbar{\mathbf{x}}),(u,\overbar{\mathbf{p}}),Q_{22}\right)&=(r+\delta)wu+P_{1,0}\overbar{\mathbf{x}}u+\langle\overbar{\mathbf{x}},A^\ast\overbar{\mathbf{p}}\rangle_{\calm_2^2}+\frac12 Q_{22}\left(P_{1,0}\overbar{\mathbf{x}}\right)^2\sigma_y\cdob\sigma_y\\
  &=(r+\delta)wu+x^{(1)}_0u+\langle\overbar{\mathbf{x}},A^\ast\overbar{\mathbf{p}}\rangle_{\calm_2^2}+\frac12 Q_{22}\left\vert x^{(1)}_0\sigma_y\right\vert^2\ ,
\end{align*}
and
\begin{equation}
  \label{Hmax}
  \H_{\text{max}}\left(x_0^{(1)},u,Q_{11},Q_{12}\right)=\sup_{\pi\in\R^n\times\R_+\times\R_+}\H_c(x_0^{(1)},u,Q_{11},Q_{12},\pi)\ ,
\end{equation}
with
\begin{align*}
  \H_c(x_0^{(1)},u,Q_{11},Q_{12},\pi)&=\left[\theta\cdob(\mu-r\mathbf{1})-c-\delta B\right]u+\frac12\left\vert\theta^\top\sigma\right\vert^2Q_{11}+\theta^\top\sigma\sigma_y x^{(1)}_0Q_{12}\\
                                     &+\frac{c^{1-\gamma}}{1-\gamma}+ \delta \frac{\big(k B\big)^{1-\gamma}}{1-\gamma}\ .
\end{align*}
Reordering the terms in the definition of $\H_c$ we can write
\begin{align*}
  \H_c(x_0^{(1)},u,Q_{11},Q_{12},\pi)&=\frac{c^{1-\gamma}}{1-\gamma}-cu+\delta \frac{\big(k B\big)^{1-\gamma}}{1-\gamma}-\delta Bu\\
                                     &+\frac12\left\vert\theta^\top\sigma\right\vert^2Q_{11}+\theta^\top\sigma\sigma_y x^{(1)}_0Q_{12}+\theta\cdob(\mu-r\mathbf{1})u
\end{align*}
from which is apparent that for each $x^{(1)}_0\in\R$ and $Q_{12}\in\R$ there are three possible situations:
\begin{enumerate}[label=$\mathbf{(\roman{*})}$]
\item\label{case1} if $u>0$ and $Q_{11}<0$ the supremum in (\ref{Hmax}) is achieved at $\left(\theta^\ast,c^\ast,B^\ast\right)$, where
\begin{equation*}
  \theta^\ast=-(\sigma\sigma^\top)^{-1}\frac{1}{Q_{11}}\left[(\mu-r\mathbf{1})u+\sigma\sigma_yx^{(1)}_0Q_{12}\right]\ ,c^\ast=u^{-\frac{1}{\gamma}}\ ,B^\ast=k^{-b}u^{-\frac{1}{\gamma}}
\end{equation*}
with
\begin{equation*}
  b=1-\frac{1}{\gamma}
\end{equation*}
or equivalently
\begin{equation*}
  b=\frac{1}{\gamma^\prime}\text{ with }\gamma^\prime=\frac{\gamma}{\gamma-1}\ ;
\end{equation*}
\item if $u<0$ or $Q_{11}>0$ then the supremum in (\ref{Hmax}) is $+\infty$;
\item if $uQ_{11}=0$ the supremum in (\ref{Hmax}) can be finite or infinite depending on $\gamma$ and on the sign of the other terms involved.
\end{enumerate}

The Hamilton-Jacobi-Bellman equation associated with Problem \ref{pbl} is the partial differential equation in the unknown $v\colon\calh\to\R$
\begin{equation}
  \label{HJB}
  (\rho+\delta)v(w,\overbar{\mathbf{x}})=\H\left(w,\overbar{\mathbf{x}},\nabla v(w,\overbar{\mathbf{x}}),\nabla^2v(w,\overbar{\mathbf{x}})\right)\ .
\end{equation}
\begin{definition}
  \label{def_sol}
  A function $v\colon\calh_{++}\to\R$ is a \emph{classical solution} of the Hamilton-Jacobi-Bellman equation (\ref{HJB}) if it satisfies:
  \begin{enumerate}[label=$(\alph{*})$]
  \item $v$ is continuously Fr\'echet differentiable in $\calh_{++}$ and its four second Fr\'echet derivatives with respect to the couple $(w,x^{(1)}_0)$ exist and are continuous in $\calh_{++}$;
  \item\label{regularity} $\partial_{\overbar{\mathbf{x}}}v(w,\overbar{\mathbf{x}})$ belongs to $\cald(A^\ast)$ for every $(w,\overbar{\mathbf{x}})\in\calh_{++}$ and $A^\ast\partial_{\overbar{\mathbf{x}}}v$ is continuous in $\calh_{++}$;
  \item $v$ satisfies (\ref{HJB}) for every $(w,\overbar{\mathbf{x}})\in\calh_{++}$.
  \end{enumerate}
\end{definition}

\begin{remark}
\label{rem:H_tilde}
  The difference between $\widetilde\H$ and $\H$ lies in the term involving $A$, that appears as $\langle A\overbar{\mathbf{x}},\overbar{\mathbf{p}}\rangle$ in the former but as $\langle\overbar{\mathbf{x}},A^\ast\overbar{\mathbf{p}}\rangle$ in the latter. This choice makes $\H$ defined on the whole $\calh_{++}$ instead than only on $\calh_{++}\cap(\R\times\cald(A))$, at the price of requiring further regularity of the solution, as specified in Definition \ref{def_sol}-~\ref{regularity}. This will not constitute a problem as we are going to find an explicit solution that satisfies the required properties.
\end{remark}

If a solution $v$ to (\ref{HJB}) satisfies $\nabla v_1=\partial_w v>0$ and $\nabla^2v_{11}=\partial^2_{ww}v<0$ uniformly in $(w,\overbar{\mathbf{x}})$, then we fall in case \ref{case1} above and, plugging $\theta^\ast,c^\ast,B^\ast$ in the definition of $\H$, we find the equation for $v$ to take the form
\begin{equation}
  \label{HJB_short}
\begin{aligned}
  (\rho+\delta)v(w,\overbar{\mathbf{x}})=&(r+\delta)w\partial_wv(w,\overbar{\mathbf{x}})+x^{(1)}_0\partial_wv(w,\overbar{\mathbf{x}})-\frac{1}{b}\partial_w(w,\overbar{\mathbf{x}})^b\left(1+\delta k^{-b}\right)\\
  &+\left\langle\overbar{\mathbf{x}},A^\ast\partial_{\overbar{\mathbf{x}}}v(w,\overbar{\mathbf{x}})\right\rangle_{\calm_2^2}+\frac12\left\vert x^{(1)}_0\sigma_y\right\vert^2\partial^2_{x^{(1)}_0x^{(1)}_0}v(w,\overbar{\mathbf{x}})\\
  &-\frac12\frac{1}{\partial^2_{ww}v(w,\overbar{\mathbf{x}})}\left[(\mu-r\mathbf{1})\partial_wv(w,\overbar{\mathbf{x}})+\sigma\sigma_yx^{(1)}_0\partial^2_{wx^{(1)}_0}v(w,\overbar{\mathbf{x}})\right]\cdob\\
  &\phantom{(r+\delta)w\partial_wv(w,\overbar{\mathbf{x}})+}\cdob(\sigma\sigma^\top)^{-1}\left[(\mu-r\mathbf{1})\partial_wv(w,\overbar{\mathbf{x}})+\sigma\sigma_yx^{(1)}_0\partial^2_{wx^{(1)}_0}v(w,\overbar{\mathbf{x}})\right]\ .
\end{aligned}
\end{equation}
Set now
\begin{equation}\label{defnufinfty}
  \nu=\frac{\gamma}{\rho+\delta-(1-\gamma)\left(r+\delta+\frac{\vert\kappa\vert^2}{2\gamma}\right)},\quad   f_\infty=(1+\delta k^{-b})\nu
\end{equation}
and define, for every $(w,\overbar{\mathbf{x}})\in\calh_{++}$,
\begin{equation}
  \label{def_v_tilde}
  \tilde{v}(w,\overbar{\mathbf{x}}):=\frac{f_\infty^\gamma}{1-\gamma}\Gamma_\infty^{1-\gamma}(w,\overbar{\mathbf{x}})\ ,
\end{equation}
where $\Gamma_\infty$ is defined in (\ref{gammainfinity}).
\begin{theorem}
  \label{thm_sol1}
  The function $\tilde{v}$ is a classical solution of the Hamilton-Jacobi-Bellman equation (\ref{HJB}).
\end{theorem}
To prove the theorem we need a brief result that we state separately for later reference.
\begin{lemma}
  \label{lemma_Dast}
  The element
  \begin{equation*}
    \overbar{\mathbf{l}}_\infty=\begin{pmatrix}(g_\infty,h_\infty)\\-i_\infty(g_\infty,h_\infty)
    \end{pmatrix}
  \end{equation*}
belongs to $\cald(A^\ast)$.
\end{lemma}
\begin{proof}
Being the integral of an $L^2$ function, $h_\infty$ is differentiable almost everywhere in $[-d,0]$ and
\begin{equation*}
  h^\prime_\infty(s)=-(r+\delta)h_\infty(s)+g_\infty\phi(s)
\end{equation*}
almost everywhere. Set for brevity
\begin{equation*}
  \beta=K_1(r+\delta)+\int_{-d}^0e^{(r+\delta)s}\phi(s){\rm d}s=r+\delta-\my-\epsilon+\sigma_y\cdob\kappa\ ;
\end{equation*}
then
\begin{equation*}
  \beta g_\infty-h_\infty(0)=1
\end{equation*}
and therefore $h_\infty$ satisfies the differential equation
\begin{equation}
\label{eq_diff}
  \begin{cases}
    h^\prime=g_\infty \phi-(r+\delta)h\\
    h(0)=\beta g_\infty-1\ .
  \end{cases}
\end{equation}
Since $1\geq e^{-2(r+\delta)(s-\tau)}>0$ on $-d\leq\tau\leq s\leq 0$ and $\phi$ is an $L^2$ function, it is easy to check that $h_\infty$ is in $L^2$ as well; this implies that actually $h_\infty\in W^{1,2}(-d,0:\R)$ and since obviously $h_\infty(-d)=0$ the claim is proved.
\end{proof}
\begin{proof}[Proof of Theorem \ref{thm_sol1}]
  Recall that $\calh_{++}$ is by definition the set where $\Gamma_\infty$ is strictly positive. Thanks to the linearity of $\Gamma_\infty$ the function $\tilde{v}$ is twice continuously Fr\'echet differentiable in all variables. The derivatives that appear in the Hamiltonian are easily computed:
  \begin{align*}
    \partial_w\tilde{v}(w,\overbar{\mathbf{x}})&=f^\gamma_\infty\Gamma_\infty^{-\gamma}(w,\overbar{\mathbf{x}}),\\
    \partial_{\overbar{\mathbf{x}}}\tilde{v}(w,\overbar{\mathbf{x}})&=f_\infty^\gamma\Gamma_\infty^{-\gamma}(w,\overbar{\mathbf{x}})\overbar{\mathbf{l}}_\infty,\\
    \partial^2_{ww}\tilde{v}(w,\overbar{\mathbf{x}})&=-\gamma f^\gamma_\infty\Gamma_\infty^{-(1+\gamma)}(w,\overbar{\mathbf{x}}),\\
    \partial^2_{wx^{(1)}_0}\tilde{v}(w,\overbar{\mathbf{x}})&=-\gamma f^\gamma_\infty\Gamma_\infty^{-(1+\gamma)}(w,\overbar{\mathbf{x}})g_\infty,\\
    \partial^2_{x^{(1)}_0x^{(1)}_0}\tilde{v}(w,\overbar{\mathbf{x}})&=-\gamma f^\gamma_\infty\Gamma_\infty^{-(1+\gamma)}(w,\overbar{\mathbf{x}})g^2_\infty\ .
  \end{align*}

Therefore thanks to Lemma \ref{lemma_Dast} also requirement \ref{regularity} in the definition of solution is satisfied. It remains to check that $\tilde{v}$ satisfies (\ref{HJB}).\\
Since $f_\infty>0$, by definition of $\calh_{++}$ we have $\partial_w\tilde{v}>0$ and $\partial^2_{ww}\tilde{v}<0$ on $\calh_{++}$, therefore we can consider the simplified form (\ref{HJB_short}) for the Hamilton-Jacobi-Bellman equation.\\
Let us now look at the various pieces appearing in (\ref{HJB_short}). We have, by simple computations,
\begin{multline*}
  (\rho+\delta)w\partial_wv(w,\overbar{\mathbf{x}})+x^{(1)}_0\partial_wv(w,\overbar{\mathbf{x}})-\frac{1}{b}\partial_w(w,\overbar{\mathbf{x}})^b\left(1+\delta k^{-b}\right)\\
  =f_\infty^\gamma\Gamma_\infty^{-\gamma}(w,\overbar{\mathbf{x}})\left[(r+d)w+x^{(1)}_0-\frac{\gamma}{\gamma-1}f_\infty^{-1}\Gamma_\infty(w,\overbar{\mathbf{x}})\left(1+\delta k^{\frac{1-\gamma}{\gamma}}\right)\right]\ ,
\end{multline*}
\begin{equation*}
  \frac12\left\vert x^{(1)}_0\sigma_y\right\vert^2\partial^2_{x^{(1)}_0x^{(1)}_0}v(w,\overbar{\mathbf{x}})=-f_\infty^\gamma\Gamma_\infty^{-\gamma}(w,\overbar{\mathbf{x}})\frac12\left\vert x^{(1)}_0\right\vert^2\left\vert\sigma_y\right\vert^2\gamma g_\infty^2 \Gamma_\infty^{-1}(w,\overbar{\mathbf{x}})\ ,
\end{equation*}
\begin{multline*}
  -\frac12\frac{1}{\partial^2_{ww}(w,\overbar{\mathbf{x}})}\left[(\mu-r\mathbf{1})\partial_wv(w,\overbar{\mathbf{x}})+\sigma\sigma_yx^{(1)}_0\partial^2_{wx^{(1)}_0}v(w,\overbar{\mathbf{x}})\right]\cdob\\
  \cdob(\sigma\sigma^\top)^{-1}\left[(\mu-r\mathbf{1})\partial_wv(w,\overbar{\mathbf{x}})+\sigma\sigma_yx^{(1)}_0\partial^2_{wx^{(1)}_0}v(w,\overbar{\mathbf{x}})\right]\\
  =f_\infty^\gamma\Gamma_\infty^{-\gamma}(w,\overbar{\mathbf{x}})\frac{1}{2\gamma}\Gamma_\infty(w,\overbar{\mathbf{x}})\left[\left\vert\kappa\right\vert^2-2\gamma x^{(1)}_0g_\infty\kappa\cdob\sigma_y \Gamma_\infty^{-1}(w,\overbar{\mathbf{x}})+\left\vert x^{(1)}_0\right\vert^2\left\vert\sigma_y\right\vert^2\gamma^2 g_\infty^2\Gamma_\infty^{-2}(w,\overbar{\mathbf{x}})\right]
\end{multline*}
and finally, using (\ref{Astar}) and (\ref{eq_diff}),
\begin{equation*}
  \left(A^\ast\partial_{\overbar{\mathbf{x}}}\tilde v(w,\overbar{\mathbf{x}})\right)_0=f_\infty^\gamma\Gamma_\infty^{-\gamma}(w,\overbar{\mathbf{x}})\begin{pmatrix}g_\infty\left(r+\delta+\sigma_y\cdob\kappa\right)-1\\i_\infty-\epsilon g_\infty-i_\infty g_\infty\left(r+\delta-\epsilon+\sigma_y\cdob\kappa\right)\end{pmatrix}
\end{equation*}
and
\begin{equation*}
  \left(A^\ast\partial_{\overbar{\mathbf{x}}}\tilde v(w,\overbar{\mathbf{x}})\right)_1=f_\infty^\gamma\Gamma_\infty^{-\gamma}(w,\overbar{\mathbf{x}})(r+\delta)\begin{pmatrix}h_\infty\\-i_\infty h_\infty\end{pmatrix}\ ,
\end{equation*}
hence
\begin{equation}
  \label{Astar_l}
\begin{aligned}
  \left\langle\overbar{\mathbf{x}},A^\ast\partial_{\overbar{\mathbf{x}}}v(w,\overbar{\mathbf{x}})\right\rangle_{\calm_2^2}&=f_\infty^\gamma\Gamma_\infty^{-\gamma}(w,\overbar{\mathbf{x}})x^{(1)}_0\left(g_\infty\left(r+\delta+\sigma_y\cdob\kappa\right)-1\right)\\
                                                                                                                          &\phantom{=}+f_\infty^\gamma\Gamma_\infty^{-\gamma}(w,\overbar{\mathbf{x}})x^{(2)}_0\left(i_\infty-\epsilon g_\infty-i_\infty g_\infty\left(r+\delta-\epsilon+\sigma_y\cdob\kappa\right)\right)\\
                                                                                                                          &\phantom{=}+f_\infty^\gamma\Gamma_\infty^{-\gamma}(w,\overbar{\mathbf{x}})(r+\delta)\left\langle x^{(1)}_1,h_\infty\right\rangle-(r+\delta)i_\infty\left\langle x^{(2)}_1,h_\infty\right\rangle\ .
\end{aligned}
\end{equation}
Plugging now everything into (\ref{HJB_short}) and multiplying both sides by $f_\infty^{-\gamma}\Gamma_\infty^{\gamma}(w,\overbar{\mathbf{x}})$ (which is a positive quantity on $\calh_{++}$ by Assumption \ref{Hyp_gamma}) we find
\begin{align*}
  \frac{\rho+\delta}{1-\gamma}\Gamma_\infty(w,\overbar{\mathbf{x}})&=(r+\delta)w-\frac{\gamma}{\gamma-1}f_\infty^{-1}\Gamma_\infty(w,\overbar{\mathbf{x}})\left(1+\delta k^{\frac{1-\gamma}{\gamma}}\right)+\frac{1}{2\gamma}\vert\kappa\vert^2 \Gamma_\infty(w,\overbar{\mathbf{x}})+\left(r+\delta\right)x^{(1)}_0g_\infty\\
                                                                 &\phantom{=}+x^{(2)}_0i_\infty-x^{(2)}_0\epsilon g_\infty-x^{(2)}_0i_\infty g_\infty (r+\delta)+x^{(2)}_0i_\infty g_\infty(\epsilon-\sigma_y\cdob\kappa)\\
                                                                 &\phantom{=}+(r+\delta)\left\langle x^{(1)}_1,h_\infty\right\rangle-(r+\delta)i_\infty\left\langle x^{(2)}_1,h_\infty\right\rangle\\
                                                                &=(r+\delta)\Gamma_\infty(w,\overbar{\mathbf{x}})-\frac{\gamma}{\gamma-1}f_\infty^{-1}\Gamma_\infty(w,\overbar{\mathbf{x}})\left(1+\delta k^{\frac{1-\gamma}{\gamma}}\right)+\frac{1}{2\gamma}\vert\kappa\vert^2 \Gamma_\infty(w,\overbar{\mathbf{x}})\\
                                                                &\phantom{=}+x^{(2)}_0\big(i_\infty-\epsilon g_\infty+i_\infty g_\infty (\epsilon-\sigma_y\cdob\kappa)\big)
\end{align*}
but
\begin{align}
\nonumber  i_\infty-\epsilon g_\infty+i_\infty g_\infty (\epsilon-\sigma_y\cdob\kappa)&=\frac{\epsilon}{K_2}-\frac{\epsilon}{K_1}+\frac{\epsilon}{K_1K_2}(\epsilon-\sigma_y\cdob\kappa)\\
\label{comb_0}                                                                       &=\frac{\epsilon}{K_1K_2}(K_1-K_2+\epsilon-\sigma_y\cdob\kappa)=0\
\end{align}
therefore, dividing by the positive quantity $\Gamma_\infty(w,\overbar{\mathbf{x}})$ we obtain eventually
\begin{equation*}
  \frac{\rho+\delta}{1-\gamma}=(r+\delta)-\frac{\gamma}{\gamma-1}f_\infty^{-1}\left(1+\delta k^{\frac{1-\gamma}{\gamma}}\right)+\frac{1}{2\gamma}\vert\kappa\vert^2\
\end{equation*}
and this last equality is easily shown to hold true by the definition of $f_\infty$.
\end{proof}

\section{Solution of the general  problem}

\subsection{The admissible paths at the boundary}
Fix $(w,\overbar{\mathbf{x}})\in\calh_+$ and $\pi\in\Pi(w,\overbar{\mathbf{x}})$ and let $\calx(\cdot;\pi)=\left(W(\cdot;\pi),\overbar{\mathbf{Y}}(\cdot)\right)$ be the corresponding solution of (\ref{systemH_condensed}). Applying the Ito formula proved in \cite[Proposition 1.165]{FABBRI_GOZZI_SWIECH_BOOK} to the process $\langle\overbar{\mathbf{l}}_\infty,\overbar{\mathbf{Y}}\rangle_{\calm_2^2}$ and using (\ref{Astar_l}), (\ref{comb_0}) and (\ref{DEF_KAPPA}) we obtain
\begin{equation}
  \label{ito_l_infty}
  \begin{aligned}
{\rm d}\langle\overbar{\mathbf{l}}_\infty,&\overbar{\mathbf{Y}}(t)
\rangle_{\calm_2^2}=\langle A^\ast\overbar{\mathbf{l}}_\infty,\overbar{\mathbf{Y}}(t)
\rangle_{\calm_2^2}{\rm d}t+\langle\overbar{\mathbf{l}}_\infty,
F\left(\overbar{\mathbf{Y}}(t)\right)\cdob{\rm d}Z(t)\rangle_{\calm_2^2}\\
&=y(t)\left(g_\infty(r+\delta+\sigma_y\cdob\kappa)-1\right){\rm d}t+e(t)\left(i_\infty-\epsilon g_\infty-i_\infty g_\infty(r+\delta-\epsilon+\sigma_y\cdob\kappa)\right){\rm d}t\\
&\phantom{=}+(r+\delta)\langle y(t+\cdot),h_\infty\rangle{\rm d}t-(r+\delta)i_\infty\langle e(t+\cdot),h_\infty\rangle{\rm d}t+g_\infty y(t)\sigma_y\cdob{\rm d}Z(t);
    \end{aligned}
  \end{equation}
therefore setting
\begin{equation}
\label{def_Gamma_proc}
\overline\Gamma_\infty(t):=\Gamma_\infty\left(W(t;\pi),\overbar{\mathbf{Y}}(t)\right)
\end{equation}
we have
\begin{align}
\nonumber  {\rm d}\overline\Gamma_\infty(t)&=y(t)\left(g_\infty(r+\delta+\sigma_y\cdob\kappa)-1\right){\rm d}t+e(t)\left(i_\infty-\epsilon g_\infty-i_\infty g_\infty(r+\delta-\epsilon+\sigma_y\cdob\kappa)\right){\rm d}t\\
\nonumber                         &\phantom{=}+(r+\delta)\langle y(t+\cdot),h_\infty\rangle{\rm d}t-(r+\delta)i_\infty\langle e(t+\cdot),h_\infty\rangle{\rm d}t+g_\infty y(t)\sigma_y\cdob{\rm d}Z(t)\\
\nonumber                         &\phantom{=}+(r+\delta)W(t)+\left(\theta(t)\cdob(\mu-r\1)-c(t)-\delta B(t)\right){\rm d}t+y(t){\rm d}t+\theta(t)\cdob\sigma{\rm d}Z(t)\\
\label{dGamma}                                                                                     &=(r+\delta)\overline\Gamma_\infty(t){\rm d}t-\left(c(t)+\delta B(t)\right){\rm d}t+\left(g_\infty y(t)\sigma_y+\sigma^\top\theta(t)\right)\cdob\left(\kappa{\rm d}t+{\rm d}Z(t)\right)\ .
\end{align}
In what follows we will denote by $\tau_+$ the first exit time of $\calx(\cdot;\pi)$ from $\calh_{++}$:
\begin{equation}
  \label{eq:tau_+}
  \tau_+=\inf\left\{t\geq 0\colon \calx\in\calh_{++}^\complement\right\}=\inf\left\{t\geq 0\colon \calx\in\partial\calh_+\right\}=\inf\left\{t\geq 0\colon\overline\Gamma_\infty(t)=0\right\}.
\end{equation}

We can then prove the following result on the behavior of the process $\overbar{\Gamma}_\infty$ when it hits the boundary of $\calh_+$. The proof is postponed to the Appendix.
\begin{proposition}
\label{prop_tau}
  Let $(w,\overbar{\mathbf{x}})\in\calh_+$;
 \begin{enumerate}[label=$(\roman{*})$]
 \item if $\Gamma_\infty(w,\overbar{\mathbf{x}})=\overline\Gamma_\infty(0)=0$ then $\P$-a.s. $\overline\Gamma_\infty(t)=0$ for every $t>0$ and
   \begin{equation*}
     c(t,\omega)=B(t,\omega)=0,\quad g_\infty y(t,\omega)\sigma_y+\sigma^\top\theta(t,\omega)=0
   \end{equation*}
${\rm d}t\otimes\P$-a.e. on $[0,+\infty)\times\Omega$;
   \item if $\Gamma_\infty(w,\overbar{\mathbf{x}})>0$ (i.e. $\overline\Gamma_\infty(0)\in\calh_{++}$) then $\P$-a.s. for every $t\geq 0$
     \begin{equation*}
       \ind_{(\tau_+,+\infty)}(t)\overline\Gamma_\infty(t)=0
     \end{equation*}
and
\begin{equation*}
  \ind_{(\tau_+,+\infty)}(t)c(t,\omega)=\ind_{(\tau_+,+\infty)}(t)B(t,\omega)=0,\quad \ind_{(\tau_+,+\infty)}(t)\left(g_\infty y(t,\omega)\sigma_y+\sigma^\top\theta(t,\omega)\right)=0
\end{equation*}
${\rm d}t\otimes\P$-a.e. on $[0,+\infty)\times\Omega$.
 \end{enumerate}
\end{proposition}


\subsection{Fundamental identity}

In this subsection we assume $\gamma\in(0,1)$;
first we state a key lemma (proved in the Appendix) to deal with the infinite horizon nature of the problem.
\begin{lemma}
\label{lemma_mean_zero}
  Assume $(w,\overbar{\mathbf{x}})\in\calh_{++}$ and $\pi\in\Pi(w,\overbar{\mathbf{x}})$. Let $\gamma \in (0,1)$. Then for every $T>0$
  \begin{equation*}
    \E\left[e^{(\gamma-1)\left(r+\delta+\frac{\vert\kappa\vert^2}{2\gamma}\right)\left(T\wedge\tau_+\right)}\tilde{v}\left(\calx(T\wedge\tau_+;\pi)\right)\right]\leq \tilde{v}(w,\overbar{\mathbf{x}})\ .
  \end{equation*}
Moreover
\begin{equation*}
\lim_{T\to+\infty}
\E\left[e^{-(\rho+\delta)
\left(T\wedge\tau_+\right)}
\tilde{v}\left(\calx(T\wedge\tau_+;\pi)\right)\right]=0\ .
\end{equation*}
\end{lemma}

The key step to the main result of our paper is provided by the following result.
\begin{proposition}
\label{prop:fundamental}
Assume $(w,\overbar{\mathbf{x}})\in\calh_{++}$ and $\pi\in\Pi(w,\overbar{\mathbf{x}})$. Then
\begin{equation}
\label{eq:vtilde_J}
\begin{aligned}
\tilde{v}\left(w,\overbar{\mathbf{x}}\right)&
=J\left(w,\overbar{\mathbf{x}},\pi\right)\\
&\phantom{=}+\E\int_0^{\tau_+}e^{-(\rho+\delta)s}\left\{\H_{\text{max}}
\left(y(s),\partial_w\tilde{v}\left(\calx(s;\pi)\right),
\partial_{ww}\tilde{v}\left(\calx(s;\pi)\right),
\partial_{wx^{(1)}_0}\tilde{v}\left(\calx(s;\pi)\right)\right)\right.\\
&\phantom{=}\left.-\H_c\left(y(s),\partial_w\tilde{v}\left(\calx(s;\pi)\right),\partial_{ww}\tilde{v}\left(\calx(s;\pi)\right),\partial_{wx^{(1)}_0}\tilde{v}\left(\calx(s;\pi)\right),\pi\right)\right\}{\rm d}s\ .
\end{aligned}
\end{equation}
\end{proposition}
Identity (\ref{eq:vtilde_J}) is often called the \emph{fundamental identity}.
\begin{proof}
Set
  \begin{equation}
  \label{tau_N}
  \tau_N:=\inf\left\{t\geq 0 \colon \overline\Gamma_\infty(t)\leq\frac1N\right\}\ .
\end{equation}
and choose $N$ large enough so that $\tau_N>0$ almost surely. Ito formula applied on $\left[0,\tau_N\right]$ to $e^{-(\rho+\delta)s}f^\gamma_\infty\overline\Gamma_\infty^{1-\gamma}(s)$ yields, by (\ref{dGamma}),
  \begin{align*}
    \frac{1}{1-\gamma}&{\rm d}\left(e^{-(\rho+\delta)s}f^\gamma_\infty
    \overline\Gamma_\infty^{1-\gamma}(s)\right)
    =e^{-(\rho+\delta)s}\left\{-(\rho+\delta)f_\infty^\gamma
    \overline\Gamma_\infty^{1-\gamma}(s){\rm d}s+f_\infty^\gamma(1-\gamma)\overline\Gamma_\infty^{-\gamma}(s){\rm d}\overline\Gamma_\infty(s)\phantom{\frac{1}{2}}\right.
    \\
    &\phantom{e^{-(\rho+\delta)s}\{-(\rho+\delta)f_\infty^\gamma
    \overline\Gamma_\infty^{1-\gamma}}\left.
    -\frac12\gamma(1-\gamma)f_\infty^\gamma
    \overline\Gamma_\infty(s)^{-\gamma-1}{\rm d}\left[\overline\Gamma_\infty\right](s)\right\}
    \\
    &=e^{-(\rho+\delta)s}\left\{-(\rho+\delta)
    \tilde{v}\left(\calx(s;\pi)\right){\rm d}s+\partial_w\tilde{v}\left(\calx(s;\pi)\right)\left[(r+\delta)
    \overline\Gamma_\infty(s)+g_\infty y(s)\sigma_y\cdob\kappa\right]{\rm d}s\phantom{\frac12}\right.
    \\
    &\phantom{=}+\frac12\partial^2_{x^{(1)}_0 x^{(1)}_0}\left\vert g_\infty y(s)\sigma_y\right\vert^2{\rm d}s+f_\infty^\gamma\overline\Gamma_\infty^{-\gamma}(s)\left(g_\infty y(s)\sigma_y+\sigma^\top\theta(s)\right)\cdob{\rm d} Z(s)
    \\
   &\phantom{=}\left.+\partial_w\tilde{v}\left(\calx(s;\pi)\right)
   \left[-c(s)-\delta B(s)+\theta(s)\cdob\sigma\kappa
   -\frac12\gamma\overline\Gamma_\infty^{-1}(s)
   \left(\theta(s)^\top\sigma\sigma^\top\theta(s)+2g_\infty y(s) \sigma_y\cdob\sigma^\top\theta(s)\right)\right]{\rm d}s\right\}
   \\
   &=e^{-(\rho+\delta)s}\left\{-(\rho+\delta)
   \tilde{v}\left(\calx(s;\pi)\right){\rm d}s
   +\partial_w\tilde{v}\left(\calx(s;\pi)\right)
   \left[(r+\delta)W(s;\pi)+(r+\delta)\left
   \langle\overbar{\mathbf{l}}_\infty,\overbar{\mathbf{Y}}(s)
   \right\rangle_{\calm_2^2}\right]{\rm d}s\phantom{\frac12}\right.
   \\
   &\phantom{=}+\partial_w\tilde{v}\left(\calx(s;\pi)\right)\left[-y(s)+y(s)+g_\infty y(s)\sigma_y\cdob\kappa\right]{\rm d}s
   \\
   &\phantom{=}+\frac12\partial^2_{x^{(1)}_0 x^{(1)}_0}\tilde{v}\left(\calx(s;\pi)\right)\left\vert y(s)\sigma_y\right\vert^2{\rm d}s+f_\infty^\gamma\overline\Gamma_\infty^{-\gamma}(s)\left(g_\infty y(s)\sigma_y+\sigma^\top\theta(s)\right)\cdob{\rm d} Z(s)
   \\
   &\phantom{=}\left.+\partial_w\tilde{v}\left(\calx(s;\pi)\right)\left[-c(s)-\delta B(s)+\theta(s)\cdob\sigma\kappa-\frac12\gamma\overline\Gamma_\infty^{-1}(s)\left(\theta(s)^\top\sigma\sigma^\top\theta(s)+2g_\infty y(s)\sigma_y\cdob\sigma^\top\theta(s)\right)\right]{\rm d}s\right\}
   \\
   &=e^{-(\rho+\delta)s}\left\{-(\rho+\delta)\tilde{v}\left(\calx(s;\pi)\right){\rm d}s+(r+\delta)W(s;\pi)\partial_w\tilde{v}\left(\calx(s;\pi)\right){\rm d}s+\phantom{\frac12}\right.
   \\
   &\phantom{=}+\left\langle\overbar{\mathbf{Y}},A^\ast\partial_{\overbar{\mathbf{x}}}\tilde{v}\left(\calx(s;\pi)\right)\right\rangle{\rm d}s+y(s)\partial_w\tilde{v}\left(\calx(s;\pi)\right){\rm d}s
   \\
    &\phantom{=}+\frac12\partial^2_{x^{(1)}_0 x^{(1)}_0}\tilde{v}\left(\calx(s;\pi)\right)\left\vert y(s)\sigma_y\right\vert^2{\rm d}s+f_\infty^\gamma\overline\Gamma_\infty^{-\gamma}(s)\left(g_\infty y(s)\sigma_y+\sigma^\top\theta(s)\right)\cdob{\rm d} Z(s)\\
                      &\phantom{=}+\partial_w\tilde{v}\left(\calx(s;\pi)\right)\left[-c(s)-\delta B(s)+\theta(s)\cdob(\mu-r\1)\right]+\partial^2_{w x^{(1)}}\tilde{v}\left(\calx(s;\pi)\right)y(s)\theta(s)\cdob\sigma\sigma_y{\rm d}s\\
                      &\phantom{=}\left.+\frac12\partial^2_{ww}\tilde{v}\left(\calx(s;\pi)\right)\theta(s)^\top\sigma\sigma^\top\theta(s){\rm ds}\right\}\\
                      &=e^{-(\rho+\delta)s}\left\{-\H_{\text{max}}\left(y(s),\partial_w\tilde{v}\left(\calx(s;\pi)\right),\partial_{ww}\tilde{v}\left(\calx(s;\pi)\right),\partial_{wx^{(1)}_0}\tilde{v}\left(\calx(s;\pi)\right)\right){\rm d}s\right.\\
                      &\phantom{=}+\H_c\left(y(s),\partial_w\tilde{v}\left(\calx(s;\pi)\right),\partial_{ww}\tilde{v}\left(\calx(s;\pi)\right),\partial_{wx^{(1)}_0}\tilde{v}\left(\calx(s;\pi)\right),\pi\right){\rm d}s\\
                      &\phantom{=}\left.+\left[\frac{c(s)^{1-\gamma}}{1-\gamma}+\delta\frac{(kB)^{1-\gamma}}{1-\gamma}\right]{\rm d}s+f_\infty^\gamma\overline\Gamma_\infty^{-\gamma}(s)\left(g_\infty y(s)\sigma_y+\sigma^\top\theta(s)\right)\cdob{\rm d} Z(s)\right\}
  \end{align*}
  where to obtain the last three equalities we used the definition of $\overline\Gamma_\infty$ as in (\ref{gammainfinity})-\eqref{def_Gamma_proc} first, then (\ref{DEF_KAPPA}), (\ref{Astar_l}) and (\ref{comb_0}) and finally the definitions of $\H_0$, $\H_{\text{max}}$ and $\H_c$, together with the derivatives of $\tilde{v}$ as computed in the proof of Theorem \ref{thm_sol1}.\\
 For $T\geq 0$ we now integrate on $\left[0,T\wedge\tau_N\right]$ and take expectation. The stochastic integral obtained integrating the last term in the chain of equalities above is a martingale. In fact, it is easy to verify that the stochastic integral is a local martingale w.r.t. the sequence of stopping times $\tau_N$ as defined in \eqref{tau_N}. We find
  \begin{align*}
    \E\Big[e^{-(\rho+\delta)\left(T\wedge \tau_N\right)}&\tilde{v}\left(\calx\left(T\wedge\tau_N;\pi\right)\right)\Big]-\tilde{v}\left(w,\overbar{\mathbf{x}}\right)\\
                                                       &=-\E\int_0^{T\wedge\tau_N}e^{-(\rho+\delta)s}\left\{\H_{\text{max}}\left(y(s),\partial_w\tilde{v}\left(\calx(s;\pi)\right),\partial_{ww}\tilde{v}\left(\calx(s;\pi)\right),\partial_{wx^{(1)}_0}\tilde{v}\left(\calx(s;\pi)\right)\right)\right.\\
                      &\phantom{=}\left.-\H_c\left(y(s),\partial_w\tilde{v}\left(\calx(s;\pi)\right),\partial_{ww}\tilde{v}\left(\calx(s;\pi)\right),\partial_{wx^{(1)}_0}\tilde{v}\left(\calx(s;\pi)\right),\pi\right)\right\}{\rm d}s\\
                      &\phantom{=}-\E\int_0^{T\wedge\tau_N}e^{-(\rho+\delta)s}\left[\frac{c(s)^{1-\gamma}}{1-\gamma}+\delta\frac{(kB)^{1-\gamma}}{1-\gamma}\right]{\rm d}s\ .
  \end{align*}
Taking now the limit $N\to+\infty$ we can use the theorem on the first expectation and the monotone convergence theorem to the terms on the right hand side (as the integrands are nonnegative almost surely) to obtain
\begin{align*}
  \tilde{v}\left(w,\overbar{\mathbf{x}}\right)&=\E\int_0^{T\wedge\tau_+}e^{-(\rho+\delta)s}\left[\frac{c(s)^{1-\gamma}}{1-\gamma}+\delta\frac{(kB)^{1-\gamma}}{1-\gamma}\right]{\rm d}s\\
                                                       &\phantom{=}+\E\int_0^{T\wedge\tau_+}e^{-(\rho+\delta)s}\left\{\H_{\text{max}}\left(y(s),\partial_w\tilde{v}\left(\calx(s;\pi)\right),\partial_{ww}\tilde{v}\left(\calx(s;\pi)\right),\partial_{wx^{(1)}_0}\tilde{v}\left(\calx(s;\pi)\right)\right)\right.\\
                      &\phantom{=}\left.-\H_c\left(y(s),\partial_w\tilde{v}\left(\calx(s;\pi)\right),\partial_{ww}\tilde{v}\left(\calx(s;\pi)\right),\partial_{wx^{(1)}_0}\tilde{v}\left(\calx(s;\pi)\right),\pi\right)\right\}{\rm d}s\\
                                              &\phantom{=}+\E\Big[e^{-(\rho+\delta)\left(T\wedge \tau_+\right)}\tilde{v}\left(\calx\left(T\wedge\tau_N;\pi\right)\right)\Big].
\end{align*}
We finally take the limit $T\to+\infty$ and use the monotone convergence Theorem and Lemma \ref{lemma_mean_zero} to obtain
\begin{align*}
  \tilde{v}\left(w,\overbar{\mathbf{x}}\right)&=\E\int_0^{\tau_+}e^{-(\rho+\delta)s}\left[\frac{c(s)^{1-\gamma}}{1-\gamma}+\delta\frac{(kB)^{1-\gamma}}{1-\gamma}\right]{\rm d}s\\
                                                       &\phantom{=}+\E\int_0^{\tau_+}e^{-(\rho+\delta)s}\left\{\H_{\text{max}}\left(y(s),\partial_w\tilde{v}\left(\calx(s;\pi)\right),\partial_{ww}\tilde{v}\left(\calx(s;\pi)\right),\partial_{wx^{(1)}_0}\tilde{v}\left(\calx(s;\pi)\right)\right)\right.\\
                      &\phantom{=}\left.-\H_c\left(y(s),\partial_w\tilde{v}\left(\calx(s;\pi)\right),\partial_{ww}\tilde{v}\left(\calx(s;\pi)\right),\partial_{wx^{(1)}_0}\tilde{v}\left(\calx(s;\pi)\right),\pi\right)\right\}{\rm d}s\ ,
\end{align*}
where the right hand side is finite because the left hand side is. To conclude just notice that by definition of $\tau_+$ and Proposition \ref{prop_tau} we have
\begin{equation*}
  \E\int_0^{\tau_+}e^{-(\rho+\delta)s}\left[\frac{c(s)^{1-\gamma}}{1-\gamma}+\delta\frac{(kB)^{1-\gamma}}{1-\gamma}\right]{\rm d}s=J\left(w,\overbar{\mathbf{x}},\pi\right)\ .
\end{equation*}

\end{proof}

We introduce the \emph{value function} $V\colon\calh\to\overbar{\R}$ defined as
\begin{equation*}
  V\left(w,\overbar{\mathbf{x}}\right)\colon=\sup_{\pi\in\Pi(w,\overbar{\mathbf{x}})}J\left(w,\overbar{\mathbf{x}};\pi\right)\ ;
\end{equation*}
note that we allow at this point $V$ to take the values $+\infty$ or $-\infty$.
\begin{corollary}
  \label{coro:V_finite}
  The value function is finite on $\calh_+$ and $V\left(w,\overbar{\mathbf{x}}\right)\leq\tilde{v}\left(w,\overbar{\mathbf{x}}\right)$ for every $\left(w,\overbar{\mathbf{x}}\right)\in\calh_+$.
\end{corollary}
\begin{proof}
  By definition of $\H_{\text{max}}$ we have that the integrand in (\ref{eq:vtilde_J}) is always nonegative, therefore $\tilde{v}\left(w,\overbar{\mathbf{x}}\right)\geq J\left(w,\overbar{\mathbf{x}},\pi\right)$ for every $\left(w,\overbar{\mathbf{x}}\right)$ and every $\pi\in\Pi\left(w,\overbar{\mathbf{x}}\right)$, so that the claim follows from the definition of the value function.
\end{proof}

\begin{remark}\label{rm:newuniqueness}
Observe that, from the proof the above Proposition \ref{prop:fundamental}, we easily obtain that the fundamental identity (\ref{eq:vtilde_J}) holds when, in place of $\bar v$, we put any classical solution $v$ of the Hamilton-Jacobi-Bellman equation (\ref{HJB}) which satisfies ($\tau$ being the first exit time from $\calh_{++}$),
\begin{equation}\label{eq:trasvnew}
\lim_{T\to + \infty}\E\left[e^{-(\rho + \delta) (T \wedge \tau_+) }
v\big(W_{\pi}(T \wedge \tau_+), X(T \wedge \tau)\big) \right]=0.
\end{equation}
Hence the same observation made in \cite[Remark 4.13]{BGP} still hold in this case.
\end{remark}

We actually aim to show that $V=\tilde{v}$ on $\calh_+$. In doing so we will also provide optimal feedback strategies.
\begin{definition}\label{DEF_ADMISSIBLE_FEEDBACK STRATEGY_INF_RET}
Fix $(w,\overbar{\mathbf{x}}) \in \calh_+$. A strategy $\tilde \pi:=\left(\tilde c,\tilde B, \tilde \theta \right)$
is called an \textit{optimal strategy} if $\tilde\pi \in \Pi\left(w,\overbar{\mathbf{x}}\right)$ and
\begin{align}
V\left(w,\overbar{\mathbf{x}}\right)=J\left(w,\overbar{\mathbf{x}},\tilde \pi\right)\ ,
\end{align}
that is, the supremum in Problem \ref{pbl} is achieved at $\tilde\pi$\ .
\end{definition}

\begin{definition}
We say that a function
$\left( \bm{\theta},\mathbf{c}, \mathbf{B}\right): \calh_+ \longrightarrow \mathbb R^n\times \mathbb R_{+}\times \mathbb R_+$
is an \textit{optimal feedback map} if for every $(w,\overbar{\mathbf{x}})\in \calh_+$
the closed loop equation
\begin{equation*}
\begin{cases}
{\rm d}W(t) =  \left[ (r+\delta) W(t)+\bm{\theta}\left( W(t), \overbar{\mathbf{Y}}(t)\right)\cdob(\mu-r\1)  +  P_{1,0}\overbar{\mathbf{Y}}(t) - \mathbf{c}\left( W(t), \overbar{\mathbf{Y}}(t)\right)\right.\\
\phantom{{\rm d}W(t) =}  \left. - \delta\mathbf{B}\left( W(t), \overbar{\mathbf{Y}}(t)\right)\right] {\rm d}t+  \bm{\theta} \left( W(t), \overbar{\mathbf{Y}}(t)\right)\cdob \sigma {\rm d}Z(t),\\
  {\rm d}\overbar{\mathbf{Y}}(t)=A\overbar{\mathbf{Y}}(t){\rm d}t+F\left(\overbar{\mathbf{Y}}(t)\right){\rm d}Z(t),\\
  \left(W(0),\overbar{\mathbf{Y}}(0)\right)=\left(w,\overbar{\mathbf{x}}\right)
\end{cases}
\end{equation*}
has a unique solution $(W^\ast,\overbar{\mathbf{Y}})=:\calx^\ast$, and the associated control strategy $\left(\tilde c, \tilde B, \tilde \theta \right)$
\begin{align}
\begin{split}
\left\{\begin{array}{ll}
\tilde c(t)&:=  \mathbf{c}\left(W^\ast(t),\overbar{\mathbf{Y}}(t)\right) {,} \\
\tilde B(t)&:=\mathbf{B}\left( W^\ast(t),\overbar{\mathbf{Y}}(t)\right){,}\\
\tilde \theta(t)&:=\bm{\theta}    \left( W^\ast(t),\overbar{\mathbf{Y}}(t)\right)
\end{array}\right. \end{split}
\end{align}
is an optimal strategy.
\end{definition}

In the Hamilton-Jacobi-Bellman equation (\ref{HJB}), the role of the variables $u,Q_{11},Q_{12}$ in the definition of $\H$ is played by, respectively,  $\partial_wv\left(w,\overbar{\mathbf{x}}\right),\partial^2_{ww}v\left(w,\overbar{\mathbf{x}}\right),\partial^2_{w x^{(1)}_0}v\left(w,\overbar{\mathbf{x}}\right)$, where $v$ is the unknown. Thus, recalling what we called case \ref{case1} after we introduced the Hamiltonian, it makes sense to define the maps
\begin{equation}\label{EQ_DEF_FEEDBACK_MAP}
\begin{cases}
  \mathbf{c}_{f}( w,\overbar{\mathbf{x}}):=   f_{\infty}^{-1}\Gamma_{\infty}( w,\overbar{\mathbf{x}}){,}    \\
\mathbf{B}_{f}( w,\overbar{\mathbf{x}}):=k^{ -b} f_{\infty}^{-1}\Gamma_{\infty}( w,\overbar{\mathbf{x}}){,}  \\
\bm{\theta}_f( w,\overbar{\mathbf{x}}):=(\sigma\sigma^\top)^{-1} (\mu-r\mathbf 1) \frac{\Gamma_{\infty} ( w,\overbar{\mathbf{x}})}{\gamma  }-   (\sigma\sigma^\top)^{-1} \sigma   \sigma_y  g_{\infty} x^{(1)}_0
\\[1.5mm]
\phantom{\bm{\theta}_f( w,\overbar{\mathbf{x}}):}=\frac{1}{\gamma  }\Gamma_{\infty} ( w,\overbar{\mathbf{x}})(\sigma^\top)^{-1} \kappa- g_{\infty} x^{(1)}_0(\sigma^\top)^{-1}  \sigma_y\ {.}
\end{cases}
\end{equation}
We want to prove that this is an optimal feedback map.\\
For given $\left(w,\overbar{\mathbf{x}}\right)\in \calh_+$, denote with $W_f^\ast(t)$ the unique solution of the associated closed loop equation
\begin{equation}
  \label{CLOSED_LOOP_EQUATION_W}
  \begin{cases}
  {\rm d}W(t) =  \left[ (r+\delta) W(t)+ \bm{\theta}_f\left(W(t),\overbar{\mathbf{x}}(t)\right)\cdob(\mu-r\1)  +  P_{1,0}\overbar{\mathbf{Y}}(t) - \mathbf{c}_f\left(W(t),\overbar{\mathbf{x}}(t)\right)\right.\\
  \phantom{{\rm d}W(t) = }\left.- \delta \mathbf{B}_f\left(W(t),\overbar{\mathbf{x}}(t)\right) \right] {\rm d}t+  \bm{\theta}_f(t) \cdob\sigma {\rm d}Z(t),\\
  {\rm d}\overbar{\mathbf{Y}}(t)=A\overbar{\mathbf{Y}}(t){\rm d}t+F\left(\overbar{\mathbf{Y}}(t)\right){\rm d}Z(t),\\
  \left(W(0),\overbar{\mathbf{Y}}(0)\right)=\left(w,\overbar{\mathbf{x}}\right)
\end{cases}
\end{equation}
and set
\begin{equation}\label{DEF_GAMMA_INFTY_STAR}
  \Gamma_{\infty}^\ast(t)= \Gamma_{\infty} \big(W_f^\ast(t), \overbar{\mathbf{Y}}(t)\big)  =W_f^\ast(t) + \langle\overbar{\mathbf{l}}_\infty,\overbar{\mathbf{Y}}(t)\rangle_{\calm_2^2}\ .
\end{equation}
The control strategy associated with (\ref{EQ_DEF_FEEDBACK_MAP}) is then
\begin{equation}
  \label{EQ_FEEDBACK_STRATEGIES_INF_RET}
\begin{cases}
\tilde c_f(t):= \mathbf{c}_f\left( W^\ast_f(t),\overbar{\mathbf{Y}}(t)\right) =  f_{\infty}^{-1}\Gamma_{\infty}^\ast(t) ,  \\
\tilde B_f(t):=\mathbf{B}_f \left( W^\ast_f(t),\overbar{\mathbf{Y}}(t)\right) =k^{ -b} f_{\infty}^{-1}\Gamma_{\infty}^\ast(t)  ,  \\
 \tilde \theta_f(t):= \bm{\theta}_f\left( W^\ast_f(t),\overbar{\mathbf{Y}}(t)\right) =\frac{\Gamma^\ast_{\infty} (t)}{\gamma  }(\sigma^\top)^{-1} \kappa- g_{\infty} P_{1,0}\overbar{\mathbf{Y}}(t)(\sigma^\top)^{-1}  \sigma_y{.}
\end{cases}
\end{equation}
We first show that this strategy is admissible.
\begin{lemma}\label{PROP_DYNAMIC_H_INF_RET}
Let $(w,\overbar{\mathbf{x}})\in \calh_+$. The process $\Gamma_{\infty}^\ast$ defined in (\ref{DEF_GAMMA_INFTY_STAR}) is a stochastic exponential satisfying equation
\begin{align}\label{DYN_H^*_INFTY}
\begin{split} {\rm d}  \Gamma_{\infty}^\ast (t) =& \Gamma_{\infty}^\ast (t) \Big(  r + \delta +\frac{1}{\gamma} \vert\kappa\vert^2- f_{\infty}^{-1}\big(1+\delta k^{-b}\big) \Big){\rm d}t+ \frac{ \Gamma_{\infty}^\ast (t)}{\gamma } \kappa\cdob {\rm d}Z(t).
\end{split}
\end{align}
\end{lemma}

\begin{proof}
Substituting \eqref{EQ_FEEDBACK_STRATEGIES_INF_RET} into the first equation of \eqref{CLOSED_LOOP_EQUATION_W} we get
\begin{align}\label{CLOSED_LOOP_W_INF_RET}
{\rm d} W_f^\ast(t) =&\Big\{ W_f^\ast(t)(r+\delta) +
 \Gamma_{\infty}^\ast(t)\left[ \frac{\vert\kappa\vert^2}{\gamma }   - f_{\infty}^{-1}\big(1 + \delta k^{-b}\big)  \right] + P_{1,0}\overbar{\mathbf{Y}}(t)-   g_{\infty} P_{1,0}\overbar{\mathbf{Y}}(t) \sigma_y\cdob\kappa\Big\}{\rm d}t\\
&+ \left\{ \frac{\Gamma_{\infty}^\ast(t)}{ \gamma} \kappa - g_{\infty} P_{1,0}\overbar{\mathbf{Y}}(t)\sigma_y \right\}\cdob {\rm d}Z(t).
\end{align}
Since
\begin{equation*}
  {\rm d}\Gamma_\infty^\ast(t)={\rm d}W^\ast_f(t)+{\rm d}\langle \overbar{\mathbf{l}}_\infty,\overbar{\mathbf{Y}}(t)\rangle_{\calm_2^2}\ ,
\end{equation*}
using (\ref{ito_l_infty}) and proceeding as in the computation leading to (\ref{dGamma}) we find the claim.
\end{proof}
We now state and prove our main result in the case $\gamma\in(0,1)$.
\begin{theorem}\label{th:VERIFICATION_THEOREM_INF_RET}
We have $V=\tilde v$ in $\calh_+$. Moreover the function $\left(\mathbf{c}_f, \mathbf{B}_f, \bm{\theta}_f\right)$ defined in (\ref{EQ_DEF_FEEDBACK_MAP}) is an optimal feedback map. Finally, for every $(w,\overbar{\mathbf{x}})\in \calh_{+}$ the strategy $\tilde\pi_f:=(\tilde c_f,\tilde B_f,\tilde\theta_f)$ is the unique optimal strategy.
\end{theorem}

\begin{proof}
We first take $(w,\overbar{\mathbf{x}}) \in \partial\calh_+=\{\Gamma_\infty=0\}$. We thus have, by equation \eqref{DYN_H^*_INFTY}, that almost surely $\Gamma^\ast_{\infty}(t)=0$ for every $t\geq 0$. This in turn implies, by \eqref{EQ_FEEDBACK_STRATEGIES_INF_RET}, that
\begin{equation*}
\tilde c_f \equiv 0,\qquad
\tilde B_f \equiv 0,\qquad
\tilde \theta_f \equiv - g_{\infty} P_{1,0}\overbar{\mathbf{Y}}(t)(\sigma^\top)^{-1}  \sigma_y {.}
\end{equation*}
It follows from Proposition \ref{prop_tau} that this is the only admissible strategy, therefore it must be optimal.\\
Now we consider $(w,\overbar{\mathbf{x}}) \in \calh_{++}$. First we observe that $\tilde\pi_f$ is an admissible strategy; indeed
by Lemma \ref{PROP_DYNAMIC_H_INF_RET} $\Gamma_{\infty}^\ast(\cdot)$ is a stochastic exponential, therefore almost surely strictly positive for any strictly positive initial condition $\Gamma_{\infty}^\ast(0)=\Gamma_\infty(w,x)$, and this implies that the constraint (\ref{constraint_rewritten}) is always satisfied, so that the strategy is admissible provided that $\tilde c_f$ and $\tilde B_f$ are nonnegative. This last fact follows however immediately from (\ref{EQ_FEEDBACK_STRATEGIES_INF_RET}).\\
Concerning optimality we observe that, as recalled above, the feedback map is obtained taking the maximum points of the Hamiltonian given in \eqref{EQ_DEF_FEEDBACK_MAP} and substituting the derivatives $\partial_w\tilde v, \partial_{ww}^2 \tilde v, \partial_{w x^{(1)}_0}\tilde v$ in place of $u,Q_{11},Q_{12}$, respectively. This implies that, substituting $\tilde\pi_f$ in the fundamental identity (\ref{eq:vtilde_J}), we obtain
\begin{equation*}
\tilde v(w,\overbar{\mathbf{x}})=J\left(w,\overbar{\mathbf{x}};\tilde\pi_f\right){.}
\end{equation*}
Hence, by Corollary \ref{coro:V_finite} and the definition of the value function,
\begin{equation*}
V(w,\overbar{\mathbf{x}})\leq \tilde v(w,\overbar{\mathbf{x}})=J\left(w,\overbar{\mathbf{x}};\tilde\pi_f\right)\leq V(w,\overbar{\mathbf{x}}){,}
\end{equation*}
which gives $ V(w,\overbar{\mathbf{x}})=J\left(w,\overbar{\mathbf{x}};\tilde\pi_f\right)$, namely, optimality of $\tilde\pi_f$.\\

We now prove uniqueness. When $(w,\overbar{\mathbf{x}})\in \partial\calh_+$ the claim again easily follows from Proposition \ref{prop_tau}. When instead $(w,\overbar{\mathbf{x}})\in \calh_{++}$, uniqueness follows from the fundamental identity (\ref{eq:vtilde_J}). Indeed, since $\tilde v=V$, if a given strategy $\pi$ is optimal at $(w,\overbar{\mathbf{x}})\in \calh_{++}$ it must satisfy $\tilde v(w,\overbar{\mathbf{x}})=J(w,\overbar{\mathbf{x}};\pi)$, which implies, substituting in (\ref{eq:vtilde_J}), that the integral in (\ref{eq:vtilde_J}) is zero. This implies that, on $[0,\tau_+]$ we have $\pi=\tilde\pi_f$, ${\rm d}t \otimes \P$-almost everywhere. This is enough for uniqueness as, for $t> \tau_+$, we still must have $\pi=\tilde\pi_f$ ${\rm d}t \otimes \P$-almost everywhere, due to Proposition \ref{prop_tau}.
\end{proof}

\begin{remark}\label{rm:gamma>1}
The result analogous to Theorem \ref{th:VERIFICATION_THEOREM_INF_RET} for the case $\gamma >1$ can be obtained using the same approach proposed in \cite[Subsection 4.6]{BGP}. We do not present this here for brevity.
\end{remark}

\section{Back to the original problem}
\label{sec:back}

We now explain what our main result (Theorem \ref{th:VERIFICATION_THEOREM_INF_RET}) on the general  Problem \ref{pbl} says on our original problem of Section 2.

First of all, from Remark \ref{rm:changevar} and Proposition \ref{equiv_stoch} we observe that our original problem
can be seen as a ``subproblem'' of the general  problem in the sense that, when the initial conditions for $y$ and $e$ satisfy the last two of \eqref{DYNAMICS_WEALTH_LABOR_INCOMEe} the optimal strategies of our general  problem is also the optimal strategies of the original problems with the initial data of $y$ in \eqref{DYNAMICS_WEALTH_LABOR_INCOMEe}.

From Theorem 5.9 and Remark \ref{rm:gamma>1}
we get the following result.

\begin{theorem}\label{th:mainlast}
	The value function $V$ of our original problem of Section 2 is given by
	\begin{equation}
	V(w,x_0,x_1) =  \frac{ f_{\infty}^{\gamma} \left(
		w + (1-i_\infty)\left[g_{\infty} x_0 + \int_{-d}^0 h_{\infty}(s) x_1(s) \,{\rm d}s\right]
		\right)^{1-\gamma} }
	{1-\gamma},
	\end{equation}
	where $f_{\infty}$ is defined in (\ref{defnufinfty})
and $\left(g_{\infty},h_{\infty},i_\infty\right)$ in \eqref{defgiinfty}-\eqref{defhinfty}.
Moreover for every $(w,x) \in \R\times \calm_2^2$
there exists a unique optimal strategy $\pi^*=(c^*,B^*,\theta^*)\in \Pi$ starting at $(w,x)$.
Such strategy can be represented as follows.
Denote total wealth by
\vspace{-0.3truecm}
\begin{equation}\label{Gamma-infty2}
  \Gamma_{\infty}^* (t): =W^*(t) + g_{\infty}(y(t)-i_\infty \E[y(t)])+ \int_{-d}^0 h_{\infty}(s) \left(y(t+s)-i_\infty\E[y(t+s)]\right) \,{\rm d}s,
\vspace{-0.3truecm}
  	\end{equation}
where $W^*(\cdot)$ is the solution of equation
\eqref{CLOSED_LOOP_EQUATION_W}
with initial datum $w$ and control $\pi^*$, whereas $y(\cdot)$ is the solution of the second equation in (\ref{DYNAMICS_WEALTH_LABOR_INCOME}) with datum $x=(x_0,x_1)\in \calm_2^2$.
Then, $\Gamma^*_\infty$ has  dynamics
	\begin{align}\label{DYN_GAMMA*_PB1}
	\begin{split}
	d  \Gamma_{\infty}^* (t) =& \Gamma_{\infty}^* (t) \Big(  r + \delta +\frac{\vert\kappa\vert^2}{\gamma}
	- f_{\infty}^{-1}\big( 1 
	+\delta k^{-b}\big) \Big){\rm d}t
	+  \frac{\Gamma_{\infty}^* (t)}{\gamma } \kappa\cdob {\rm d}Z(t),
	\end{split}
	\end{align}
and the optimal strategy triplet $\pi^*=(c^*,B^*,\theta^*)$
for our original problem of Section 2 is given by
	\begin{align}\label{OPTIMAL_STRATEGIES_RET_INF}
	\begin{split}
	c^{*}(t)&:=  
	f_{\infty}^{-1}  \Gamma_{\infty}^*(t) {,}  \\
	B^{*}(t)&:=k^{ -b } f_{\infty}^{-1}  \Gamma_{\infty}^*(t)      {,} \\
	\theta^{*}(t)&:=
\frac{\Gamma_{\infty}^*(t)}{ \gamma}
(\sigma^\top)^{-1}\kappa
- g_{\infty}y(t)(\sigma^\top)^{-1} \sigma_y.
	\end{split}
	\end{align}
\end{theorem}

In particular, it is interesting to compare the optimal solution in this case, with the one of \cite{BGP} when the mean reversion speed $\epsilon$ disappears.

First of all we observe that, due to the linear character of the infinite dimensional Merton's model, the dynamics of the optimal total capital $\Gamma^*_\infty$ is the same as in \cite{BGP} as it does not depend on $\epsilon$. However, its initial value is different since here $\Gamma^*_\infty(0)=\Gamma_\infty(w,\overbar{\mathbf{x}})$ which is now different as it contains the part coming from the mean reverting term. Moreover, as in \cite{BGP}, the dynamics of the optimal consumption and bequest are constant fractions (independent of the mean reversion speed)  of the optimal total capital $\Gamma^*_\infty$. Hence, the differences with the solution of \cite{BGP}, concerning the
optimal total capital, the optimal consumption and the optimal bequest,
come only from the initial total capital $\Gamma_\infty(w,\overbar{\mathbf{x}})$ which can be proved to be decreasing in $\epsilon$. Hence, the presence of a mean reverting term with $\epsilon <0$ makes such variables to increase their value.

On the other hand, things are different if one looks at the optimal trading strategy. This is the sum of a constant fraction of the optimal total capital $\Gamma^*_\infty$ and of a so-called negative hedging demand term.
Here such term is
$-g_{\infty} y(t)(\sigma^\top)^{-1}  \sigma_y{,}$
thus the only difference lies in the different form of
$g_{\infty}$ which, when $\epsilon <0$, is smaller than the one in the case $\epsilon =0$.

The analysis of the financial effect of benchmarking labor incomes on the life-cycle portfolio choice problem through a careful comparison between the standard Merton's model with the benchmarked one ({with or} without delay i.e. $\phi(\cdot)=0$) is also interesting and deserves a stand alone paper that  we leave for the near future.

\section*{Acknowledgments}
The authors are grateful to Enrico Biffis for useful comments and suggestions.
Boualem Djehiche gratefully acknowledge financial support by the Verg Foundation. Fausto Gozzi, Giovanni Zanco and Margherita Zanella are supported by the Italian
Ministry of University and Research (MIUR), in the framework of PRIN
projects 2015233N54 006 (Deterministic and stochastic evolution equations) and
2017FKHBA8 001 (The Time-Space Evolution of Economic Activities: Mathematical
Models and Empirical Applications).

\section*{Appendix}
\begin{proof}[Proof of Proposition \ref{prop_semigroup}.]
  \begin{enumerate}[label=$(\roman{*})$]
\item  Consider the equivalent formulation of system \eqref{system1} given by \eqref{system2} with $t_0=0$ and introduce the $2 \times2$-matrix-valued finite measure on $[-d,0]$
\begin{equation}
a({\rm d}\lambda)=C_0\delta_0({\rm d}\lambda)+\Id_{2\times 2}\phi(\lambda){\rm d}\lambda.
\end{equation}
The operator $A_0$ can be then written in the form
\begin{equation}\label{eq_gen}
A_0\left(\mathbf{x}_0,\mathbf{x}_1\right)=\left(\int_{-d}^0 \mathbf{x}_1(\lambda)a({\rm d}\lambda), \frac{{\rm d}\mathbf{x}_1}{{\rm d}s}\right),
\end{equation}
therefore it generates a strongly continuous semigroup by \cite[Proposition A.27]{DAPRATO_ZABCZYK_RED_BOOK}.
\item  The compactness property of $S(t)$ for $t$ big enough is proven for example in \cite[Chapter 7, Lemma 1.2]{HALE_VERDUYN_LUNEL_BOOK}.
\item Existence and uniqueness of a weak solution given by (\ref{semigroup_1}) for deterministic $\overbar{\mathbf{m}}$ is classical (see \cite[Proposition A.5]{DAPRATO_ZABCZYK_RED_BOOK} and the related references therein); the fact that it is actually a strong solution if $\overbar{\mathbf{m}}\in\cald(A_0)$ is proved for example in \cite[Proposition A.7]{DAPRATO_ZABCZYK_RED_BOOK}; the geralization to random $\overbar{\mathbf{m}}$ is immediate. Property (\ref{semigroup_2}) then follows from uniqueness of the solution.
\item\label{it:equiv} If $\mathbf{n}(t_0;\cdot)$ is the unique solution to (\ref{system2}) then the $\calm_2^2$-valued process $\left(\mathbf{n}(t_0;t),\mathbf{n}(t_0;t+\cdot)\right)_{t\geq t_0}$ solves (\ref{INFINITE_DIMENSIONAL_STATE_EQUATION}) by \cite[Part II, Chapter 4, Theorem 4.3]{BENSOUSSAN_DAPRATO_DELFOUR_MITTER}. Since also the latter has a unique solution, its first component must in fact be the solution to (\ref{system2}).
\item This is an immediate consequence of \ref{it:equiv}.
\end{enumerate}
\end{proof}

\begin{proof}[Proof of Lemma \ref{lemma_resolvent}.]
If $\lambda\in\R\cap R(A_0)$ then both $K_1(\lambda)$ and $K_2(\lambda)$ are nonzero, by Lemma \ref{lemma_speck}. To compute $R(\lambda, A_0)$,
we will consider for a fixed $\left(\xphi{m_0}{e_0},\xphi{m_1}{e_1}\right)\in\calm_2^2$ the equation
\begin{equation}\label{EQ_PROOF_RESOLVENT}
(\lambda-A_0)\left(\xphi{u_0}{v_0},\xphi{u_1}{v_1}\right)=\left(\xphi{m_0}{e_0},\xphi{m_1}{e_1}\right)\ ,
\end{equation}
in the unknown $\left(\xphi{u_0}{v_0},\xphi{u_1}{v_1}\right)\in\cald(A_0)$, that by definition of $A_0$ is equivalent to
\begin{equation}
\label{system3}
\left\{
\begin{aligned}
\lambda u_0- (\epsilon+\my - \sigma_y\cdob \kappa) u_0+\epsilon v_0- \int_{-d}^0u_1(\tau)\phi(\tau)\,\text{d}\tau&=m_0
\\
\lambda v_0-\my v_0-\int_{-d}^0v_1(\tau)\phi(\tau)\,\text{d}\tau =e_0
\\
\lambda u_1-\frac{\text{d}u_1}{\text{d}s}&=m_1
\\
\lambda v_1-\frac{\text{d}v_1}{\text{d}s}&=e_1.
\end{aligned}\right.
\end{equation}
Then
\[u_1(s)=e^{\lambda s} u_0+\int_s^0e^{-\lambda(s_1-s)}m_1(s_1)\, \text{d}s_1,\quad s\in[-d,0],\]
and
\[v_1(s)=e^{\lambda s} v_0+\int_s^0e^{-\lambda(s_1-s)}e_1(s_1)\, \text{d}s_1,\quad s\in[-d,0].\]
Therefore $v_0$ is determined by the equation
\[(\lambda-\my) v_0=\Big[ e_0+\int_{-d}^0\left( e^{\lambda \tau} v_0
+\int_{\tau}^0e^{-\lambda(s-\tau)}e_1(s)\, \text{d}s \right)\phi(\tau)\text{d}\tau\Big]\]
yielding
\begin{eqnarray*}
K_2(\lambda)v_0
=e_0+ \int_{-d}^0\int_{-d}^{s}e^{-\lambda (s-\tau)}\phi(\tau)\text{d}\tau \, e_1(s) \text{d} s.
\end{eqnarray*}
If we now substitute the expressions for $v_0$ and $u_1$ in the first expression of \eqref{system3}, we obtain \eqref{u0}.
\end{proof}

\begin{proof}[Proof of Proposition \ref{prop_tau}.]
Under the probability measure $\tilde{\P}_T$ (defined in (\ref{Ptilde})), the process $\overline\Gamma_\infty$ satisfies, on $[0,T]$,
\begin{equation*}
{\rm d}\overline\Gamma_{\infty} (t)=
\left[(r+\delta)\overline\Gamma_{\infty} (t)    - c(t) - \delta B(t)\right] {\rm d}t
+ \left[\sigma^\top\theta(t) +  g_{\infty}   P_{1,0}\overbar{\mathbf{Y}}(t)  \sigma_y\right]\cdob
{\rm d}\tilde Z(t).
\end{equation*}
Thus we obtain, under $\tilde\P_T$, for every
$0\le t\le T$,
\begin{equation}\label{eq:Gammatau}
\overline\Gamma_\infty(t)=e^{(r+\delta)t}
\left[\overline\Gamma_\infty(0)-\int_{0}^{t} e^{-(r+\delta)s}(c(s) + \delta B(s)){\rm d}s
+  \int_{0}^{t} e^{-(r+\delta)s}\left[\sigma^\top\theta(s) +  g_{\infty}   P_{1,0}\overbar{\mathbf{x}}(s)  \sigma_y \right]\cdob{\rm d}\tilde Z(s)\right].
\end{equation}
Setting $\overline\Gamma_\infty^0(t):=e^{-(r+\delta)t}\overline\Gamma_\infty(t)$
the above \eqref{eq:Gammatau} is rewritten as
\begin{equation}\label{eq:Gamma0tau}
\overline\Gamma^0_\infty(t)=
\overline\Gamma^0_\infty(0)-\int_{0}^{t} e^{-(r+\delta)s}(c(s) + \delta B(s)){\rm d}s
+  \int_{0}^{t} e^{-(r+\delta)s}\left[\sigma^\top\theta(s) +  g_{\infty}   P_{1,0}\overbar{\mathbf{x}}(s)  \sigma_y \right]\cdob{\rm d}\tilde Z(s){,}
\end{equation}
which implies that the process $\Gamma^0_\infty(t)$ is a supermartingale under $\tilde \P_T$ on $[0,T]$.
By the optional sampling theorem we then have, for every couple of stopping times
$0\le \tau_1\le \tau_2\le T$, denoting by $\tilde\E_T$
the expectation under $\tilde\P_T$,
\begin{equation}\label{eq:OST}
\tilde \E_T \left[\overline\Gamma^0_\infty(\tau_2)| \calf_{\tau_1}\right]
\le
\overline\Gamma^0_\infty(\tau_1), \qquad \tilde \P_T\text{-a.s.}\ .
\end{equation}
The admissibility of the strategy $\pi$, and the fact that $\P$ and $\tilde\P_T$ are equivalent on $\calf_T$, implies that
$\overline\Gamma^0_\infty(\tau_2)\ge 0$, $\tilde \P_T$-a.s., hence also
\begin{equation*}
\tilde \E_T \left[\overline\Gamma^0_\infty(\tau_2)| \calf_{\tau_1}\right]
\ge 0,\qquad \tilde \P_T\text{-a.s.}\ .
\end{equation*}
Now let $\tau_1:=\tau_+ \wedge T$ where $\tau_+$ is defined in (\ref{eq:tau_+}) (which is taken to be identically $0$
when $\Gamma_\infty(w,\overbar{\mathbf{x}})=0$). Then $\overline\Gamma^0_\infty(\tau_1)=0$ on $\{\tau_+<T\}$, and from \eqref{eq:OST} we get
\begin{equation*}
\1_{\{\tau_+<T\}}
\tilde \E_T \left[\overline\Gamma^0_\infty(\tau_2)| \calf_{\tau_1}\right]
=
\tilde \E_T \left[\overline\Gamma^0_\infty(\tau_2)\1_{\{\tau_+<T\}}| \calf_{\tau_1}\right]
=0, \qquad \tilde\P_T{\text{-a.s.}}\ .
\end{equation*}
and, consequently,
\begin{equation}\label{eq:GammaInd0}
\overline\Gamma^0_\infty(\tau_2)\1_{\{\tau_+<T\}}=0, \qquad \tilde\P_T\text{-a.s.}
\end{equation}
We now use \eqref{eq:Gamma0tau} to compute
$\overline\Gamma^0_\infty(\tau_2)-\overline\Gamma^0_\infty(\tau_1)$ getting
\begin{equation}\label{eq:Gamma0taubis}
\overline\Gamma^0_\infty(\tau_2)-\overline\Gamma^0_\infty(\tau_1)=
-\int_{\tau_1}^{\tau_2} e^{-(r+\delta)s}(c(s) + \delta B(s)){\rm d}s
+  \int_{\tau_1}^{\tau_2} e^{-(r+\delta)s}
\left[\sigma^\top\theta(s)+
g_{\infty}P_{1,0}\overbar{\mathbf{Y}}(s)\sigma_y\right]\cdob{\rm d}\tilde Z(s).
\end{equation}
Again using the optional sampling theorem we get
\begin{equation*}
\tilde \E_T \left[\overline\Gamma^0_\infty(\tau_2)| \calf_{\tau_1}\right]
-\overline\Gamma^0_\infty(\tau_1)=
-\tilde \E_T \left[\int_{\tau_1}^{\tau_2} e^{-(r+\delta)s}(c(s) + \delta B(s)){\rm d}s
| \calf_{\tau_1}\right], \qquad \tilde\P_T\text{-a.s.}
\end{equation*}
Hence, taking $\tau_2\equiv T$
\begin{equation*}
0\le  \1_{\{\tau_+<T\}}
\tilde\E_T \left[\overline\Gamma^0_\infty(T)| \calf_{\tau_1}\right]
=
-\tilde \E_T \left[\int_{0}^{T}  \1_{\{\tau_+<s\}} e^{-(r+\delta)s}(c(s) + \delta B(s)){\rm d}s
| \calf_{\tau_1}\right], \qquad \tilde\P_T\text{-a.s.}
\end{equation*}
which implies
\begin{equation}\label{eq:cB0PT}
\1_{\{\tau_+<s\}}(\omega) c(s,\omega)= \1_{\{\tau_+<s\}}(\omega)B(s,\omega)=0,
\qquad {\rm d}s \otimes \tilde\P_T\text{-a.e. in }[0,T]\times \Omega.
\end{equation}
We now multiply \eqref{eq:Gamma0taubis} by $\1_{\{\tau_+<T\}}$
and we use \eqref{eq:GammaInd0} and \eqref{eq:cB0PT} to get
\begin{equation*}
0
=
\int_{0}^{\tau_2} e^{-(r+\delta)s}\1_{\{\tau_+<s\}}
\left[\sigma^\top\theta(s)+g_{\infty}P_{1,0}
\overbar{\mathbf{Y}}(s)\sigma_y\right]\cdob{\rm d}\tilde Z(s), \qquad \tilde\P_T\text{-a.s.}
\end{equation*}

Since the integral of the right hand side is a martingale the above implies that
\begin{equation}\label{theta0PT}
\1_{\{\tau_+<s\}} \sigma^\top\theta(s) +g_{\infty}P_{1,0}\overbar{\mathbf{Y}}(s)\sigma_y=0,
\qquad {\rm d}s \otimes \tilde\P_T\text{-a.e. in }[0,T]\times \Omega.
\end{equation}
Using \eqref{eq:GammaInd0}, \eqref{eq:cB0PT},\eqref{theta0PT}, the fact that $\P$ and $\tilde \P_T$ are equivalent on $\calf_T$ and the arbitrariness of $T$ we eventually get the claim.
\end{proof}

\begin{proof}[Proof of Lemma \ref{lemma_mean_zero}.]
Since almost surely for $t<\tau_+$ we have $\Gamma_\infty\left(\calx(t;\pi)\right)>0$, we can apply the Ito formula to the process
  \begin{equation}
\label{supermartingale}
e^{(\gamma-1)\left(r+\delta+\frac{\vert\kappa\vert^2}{2\gamma}\right)t}
f^\gamma_\infty\frac{\overline\Gamma_\infty^{1-\gamma}(t)}{1-\gamma}
  \end{equation}
obtaining, by (\ref{dGamma}),
\begin{align*}
{\rm d}\bigg[e^{(\gamma-1)\left(r+\delta+\frac{\vert\kappa\vert^2}
{2\gamma}\right)t}f^\gamma_\infty
\frac{\overline\Gamma_\infty^{1-\gamma}(t)}{1-\gamma}\bigg]
&=e^{(\gamma-1)\left(r+\delta+\frac{\vert\kappa\vert^2}{2\gamma}
\right)t}f^\gamma_\infty\left\{-\overline\Gamma_\infty^{-\gamma}(t)
\left(c(t)+\delta B(t)\right)\phantom{\frac12}\right.
\\
&\phantom{=}\left.-\frac{1}{2\gamma}\overline\Gamma_\infty^{-\gamma-1}(t)
\left\vert\overline\Gamma_\infty(t)\kappa-\gamma\left(\sigma^\top\theta(t)
+g_\infty y(t)\sigma_y\right)\right\vert^2\right\}{\rm d}t
\\
&\phantom{=}+e^{(\gamma-1)\left(r+\delta+\frac{\vert\kappa\vert^2}{2\gamma}
\right)t}f^\gamma_\infty\overline\Gamma_\infty^{-\gamma}(t)
\left(\sigma^\top\theta(t)+g_\infty y(t)\sigma_y\right)\cdob{\rm d}Z(t)\ .
\end{align*}
The drift term in the equation above is negative, because $t<\tau_+$ and both $c$ and $B$ take values in $\R_+$, thus the process given by (\ref{supermartingale}) is a local $\F$-supermartingale up to the exit time $\tau_+$.\\
Set now
\begin{equation*}
  \tau_N:=\inf\left\{t\geq 0 \colon \overline\Gamma_\infty(t)\leq\frac1N\right\}\ .
\end{equation*}
Taking $N$ sufficiently large we have that $\tau_N>0$ almost surely and both the drift and the diffusion coefficients above are integrable, therefore the process
  \begin{equation*}
    e^{(\gamma-1)\left(r+\delta+\frac{\vert\kappa\vert^2}{2\gamma}\right)
    \left(T\wedge\tau_N\right)}f^\gamma_\infty
    \frac{\overline\Gamma_\infty^{1-\gamma}(T\wedge\tau_N)}{1-\gamma}
  \end{equation*}
  is in $L^1$, hence
  \begin{equation}
    \label{supermatingale_E}
    \E\left[e^{(\gamma-1)\left(r+\delta+\frac{\vert\kappa\vert^2}{2\gamma}
    \right)\left(T\wedge\tau_N\right)}f^\gamma_\infty
    \frac{\overline\Gamma_\infty^{1-\gamma}(T\wedge\tau_N)}{1-\gamma}\right]
    \leq \frac{f_\infty^\gamma}{1-\gamma}\E\left[\overline\Gamma_\infty^{1-\gamma}
    (0)\right]=\tilde{v}(w,\overbar{\mathbf{x}})\ .
  \end{equation}
  Since $\tau_N\uparrow\tau_+$ as $N\to+\infty$ and the quantity inside the expectation in the left hand side of (\ref{supermatingale_E}) is nonnegative, the first claim follows from Fatou's lemma.\\
To prove the second claim notice first that almost surely $\overline\Gamma_\infty\left(\tau_+\right)=0$ because almost surely $t\mapsto\overline\Gamma_\infty(t)$ is continuous; this implies
\begin{align*} \E\Big[e^{-(\rho+\delta)\left(T\wedge\tau_+\right)}&\tilde{v}
\left(\calx\left(T\wedge\tau_+;\pi\right)\right)\Big]=
e^{-(\rho+\delta)T}\E\left[\ind_{(\tau_+,+\infty)}(T)\tilde{v}
\left(\calx\left(T;\pi\right)\right)\right]\\ &=e^{-\left(\rho+\delta+(\gamma-1)\left(r+
\delta\frac{\vert\kappa\vert^2}{2\gamma}\right)\right)T}
\E\left[\ind_{(\tau_+,+\infty)}(T)e^{(\gamma-1)
\left(r+\delta+\frac{\vert\kappa\vert^2}{2\gamma}\right)}\tilde{v}
\left(\calx(T;\pi)\right)\right]
\\ &\leq e^{-\left(\rho+\delta+(\gamma-1)\left(r+\delta\frac{\vert\kappa\vert^2}
{2\gamma}\right)\right)T}\tilde{v}(w,\overbar{\mathbf{x}})
\end{align*}
and this last quantity converges to $0$ as $T\to+\infty$ thanks
to Assumption \ref{Hyp_gamma}.
\end{proof}

\end{document}